\documentclass[11pt,reqno,section]{amsart}

\usepackage[english]{babel}
\usepackage[utf8]{inputenc}
\usepackage{amsmath}
\usepackage{amsthm}
\usepackage{verbatim}
\usepackage{enumerate}
\usepackage{amssymb}
\usepackage{amsfonts}
\usepackage{amscd,bezier}
\usepackage{anysize}
\usepackage{indentfirst}
\usepackage{upref}
\usepackage{color}
\usepackage[pagewise]{lineno}
\usepackage{tikz}
\usepackage{theoremref}
\usetikzlibrary{matrix}
\usepackage{pstricks}
\usepackage{subfig}
\usepackage{graphicx}
\usepackage{pstricks}
\usepackage{amscd}
\usepackage{relsize}
\usepackage[toc,page]{appendix}
\usepackage{commath}

\allowdisplaybreaks

\numberwithin{equation}{section}

\newtheorem{theorem}{Theorem}[section]
\newtheorem{lemma}[theorem]{Lemma}
\newtheorem{conjecture}[theorem]{Conjecture}
\newtheorem{definition}[theorem]{Definition}
\newtheorem{proposition}[theorem]{Proposition}
\newtheorem{corollary}[theorem]{Corollary}

\newtheorem{example}[theorem]{Example}

\newtheorem{remark}[theorem]{Remark}

\marginsize{3cm}{3cm}{3cm}{3cm}

\newcommand{\NmuD}{{\mathrm{N}\mu\mathrm{D}}}
\newcommand{\sNmuD}{{\mathrm{sN}\mu\mathrm{D}}}

\newcommand{\walpha}{{\widetilde{\alpha}}}
\newcommand{\wbeta}{{\widetilde{\beta}}}
\newcommand{\wtheta}{{\widetilde{\theta}}}
\newcommand{\wnu}{{\widetilde{\nu}}}
\newcommand{\muD}{{\mu\mathrm{D}}}
\newcommand{\Nmu}{{\mathrm{N}\mu}}
\newcommand{\NED}{\mathrm{NED}}
\newcommand{\USP}{\mathrm{USP}}
\newcommand{\UPP}{\mathrm{UPP}}
\newcommand{\USPP}{\mathrm{USPP}}

\newcommand{\sNED}{\mathrm{sNED}}

\newcommand{\st}{\mathfrak{st}}
\newcommand{\un}{\mathfrak{un}}
\newcommand{\St}{\mathrm{St}}
\newcommand{\Un}{\mathrm{Un}}

\newcommand{\Si}{\Sigma}
\newcommand{\ED}{\mathrm{ED}}
\newcommand{\Id}{\mathrm{Id}}

\newcommand{\im}{\mathrm{im}}

\renewcommand{\P}{\mathrm{P}}

\newcommand{\Q}{\mathrm{Q}}

\newcommand{\sgn}{\mathrm{sgn}}
\newcommand{\rank}{\mathrm{rank}\,}

\newcommand{\qeg}{\mathrm{q}}
\newcommand{\ceg}{\mathrm{c}}

\numberwithin{equation}{section}

\allowdisplaybreaks
\usepackage{color}

\title[Discrete nonuniform $\mu$-dichotomy spectra]{Discrete $\mu$-dichotomy spectrum: beyond uniformity and new insights.}

\author[A. Casta\~{n}eda]{\'Alvaro Casta\~{n}eda $^{*}$}
\author[C. A. Gallegos]{Claudio A. Gallegos}
\address{Universidad de Chile, UCH, Facultad de Ciencias, Departamento de Matem\'aticas, Casilla 653, Santiago, Chile.}
\email{castaneda@uchile.cl,claudio.gallegos.castro@gmail.com, nestor.jara@ug.uchile.cl} 

\thanks{$^{*}$
	This author was partially supported by FONDECYT REGULAR No. 1240361. }

\author[N. Jara]{N\'estor Jara}

\date{}

\begin{document}
	
	\begin{abstract}
We develop spectral theorems for nonautonomous linear difference systems, considering different types of $\mu$-dichotomies, both uniform and nonuniform. In the nonuniform case, intriguing scenarios emerge --that have been employed but whose consequences have not been thoroughly explored-- which surprisingly exhibit unconventional behavior. These particular cases motivate us to introduce two novel properties of nonautonomous systems (even in the continuous-time framework), which appear to have been overlooked in the existing literature. Additionally, we introduce a new conceptualization of a nonuniform $\mu$-dichotomy spectrum, which lies between the traditional nonuniform $\mu$-dichotomy spectrum and the slow nonuniform $\mu$-dichotomy spectrum. Moreover, and this is particularly noteworthy, we propose a conjecture that enables the derivation of spectral theorems in this new setting. Finally, contrary to what has been believed in recent years, through the lens of optimal ratio maps, we show that the nonuniform exponential dichotomy spectrum is not preserved between systems that are weakly kinematically similar.
	\end{abstract}
	

	

	\maketitle

	\tableofcontents
	
	\section{Introduction}

Over the past two decades, the concept of nonuniform exponential dichotomy ($\NED$), introduced by Barreira and Valls \cite{BV-CMP,BV-JDE}, has emerged as a fundamental tool for studying the dynamics of nonautonomous functional, difference and differential equations, as well as nonuniform hyperbolicity in both finite and infinite dimensions, see e.g. \cite{BVD-AM,DZZ,DZZ-PL,LOO}. The motivation for defining this kind of behavior comes from ergodic theory and Lyapunov exponents, see for instance \cite[Sec. 2.2]{BV-JDDE}. This article will explore the concept of nonuniform dichotomy, along with its slight variations.

To provide context for the aim of this work, we will briefly outline some background regarding the nonuniform dichotomy spectrum: Building on Siegmund's seminal work \cite{Siegmund2002}, significant advancements have been made in the spectral theory of uniform and nonuniform exponential dichotomies. Broadly speaking, the central idea is to demonstrate that, under suitable assumptions, the dichotomy spectrum can be represented as the union of at most $d$
closed, non-overlapping intervals, where
$d$ denotes the dimension of the system being considered. 

In the literature, this type of results are commonly referred to as {\it spectral theorems}; for instance, in the context of nonautonomous linear differential equations (that is, continuous time) and nonuniform exponential dichotomies, we refer the reader to the works of Chu {\it et al.} \cite{Chu} and Zhang \cite{Xiang}, where spectral theorems have been developed. Additional details on the nonuniform exponential dichotomy spectrum can be found in \cite{APR,Kloeden}. More recently, Silva \cite{Silva} extended these results to include both exponential and more general dichotomies, as those studied on \cite{Bento,PN2, PN1,PN3,ZFY}. 

In the case of nonautonomous linear difference equations (that is, discrete time) and nonuniform exponential dichotomies, references are not so abundant. To the best of our knowledge, some of the few works address this setting are: \cite{DZZ-PL}, where the authors use a functional notion of the spectrum; in \cite{Chu2} , spectral theorems were established using the classical approach of linear integral manifolds, while in \cite{Song}, they were proven through a method based on nonuniform kinematic similarity (referred to as {\it weak kinematic similarity} in the discrete context). Despite their differing purposes (the first focused on a reducibility result, while the second applied this tool to obtain a normal forms theorem), the last two articles rely on a claim that we will later refute: The nonuniform dichotomy spectrum is preserved between systems that are weakly kinematically similar.

 \subsection{Novelty and Main Results}

Inspired by Silva's work on $\mu$-dichotomy \cite{Silva} and the limited literature on the nonuniform spectral theory in the discrete setting, the aim of this work is to explore the nonuniform $\mu$-dichotomy spectrum  of nonautonomous linear difference equations of type
\begin{equation}\label{700}
x(k+1)=A(k)x(k), \qquad k\in\mathbb{Z}, 
\end{equation}
where $A\colon\mathbb{Z}\to \mathbb{R}^{d\times d}$ is a matrix valued function. 

We provide a detailed development of new properties that emerge when considering different notions of $\mu$-dichotomies, including uniform ($\muD$), nonuniform ($\NmuD$), and slow nonuniform ($\sNmuD$), as defined in Def.~\ref{DefNmuD}. The last concept, {\it slow nonuniform dichotomy} —a term first coined in \cite{GJ}— will be of particular interest due to its unexpected behavior. As mentioned in \cite[Sec. 2.2]{BV-JDDE}, almost all linear systems may exhibit a nonuniform dichotomy, with the rate of nonuniformity becoming arbitrarily small in relation to the Lyapunov exponents. This ambiguity in the notion of ``smallness" in the rate of nonuniformity is in some sense related with the splitting we make between the notion of a slow nonuniform $\mu$-dichotomy $(\sNmuD)$ and the notion of nonuniform $\mu$-dichotomy ($\NmuD$), being both terms called with the same name so far in literature. Part of the aim of this work is to shed some light in this ambiguity. See  Rem. \ref{slowproblem} for more details.

To facilitate the reader's search for the novelty and key results of our work, we will list what is done and where it can be found in our article:

\begin{itemize}
	\item We introduce the {\it unique projector property} (UPP) and the {\it unbounded solutions property} (USP) for system \eqref{700}, see Def.~\ref{769}. We emphasize that these notions are not uniquely defined in the discrete case; they can also be extended to linear nonautonomous differential equations in continuous time.
	
	\medskip 
	
	\item If the system \eqref{700} admits $\muD$ or $\NmuD$, the UPP and USP are always fulfilled, see Lem.~\ref{705}. Surprisingly, these properties do not necessarily hold when the system \eqref{700} admits $\sNmuD$, as shown in Ex.~\ref{707} and \ref{718}.
	
	\medskip
	
	\item We state spectral theorems for $\muD$ and $\NmuD$, see Thm.~\ref{770}. Moreover, we introduce the {\it unique projector slow nonuniform $\mu$-dichotomy spectrum} for system \eqref{700} and we denote it by $\Si_{\sNmuD}^{\UPP}(A)$. This new notion of spectrum is an intermediate spectrum between the slow nonuniform $\mu$-dichotomy spectrum $\Si_{\sNmuD}(A)$ and the nonuniform $\mu$-dichotomy spectrum $\Si_{\NmuD}(A)$, see Def.~\ref{771}.
	
	\medskip
	
	\item We introduce the USPP conjecture, see Conj.~\ref{750}. Assuming this conjecture holds, we derive several intermediate results, from which we can deduce a spectral theorem for the dichotomy spectrum $\Si_{\sNmuD}^{\UPP}(A)$, see Thm.~\ref{772}. 
	
	\medskip
	
	\item With the aim of refuting the claim of invariance of the $\NmuD$ spectrum in the discrete context, we develop the {\it optimal ratio maps} introduced in \cite{GJ}, but in the discrete setting, see Sec.~\ref{optdichconst}. In the aforementioned article, the authors question the widely accepted claim that $\NmuD$ spectra remain invariant under $(\mu,\varepsilon)$-kinematic similarity for nonautonomous linear differential equations (continuous time). Based on this, a similar situation could be speculated to occur in the discrete context. In fact, we illustrate that the $\NmuD$ spectrum is not preserved for systems that are {\it weakly kinematically similar}, as asserted in \cite{Chu2} and \cite{Song}, see Ex.~\ref{773}. 
	
	\medskip
	
	\item Finally, by using the properties of the optimal ratio maps, we identify the underlying phenomenon related to the non-invariance of the dichotomy spectrum under weak kinematic similarity. This is explained in detail in Subsec.~\ref{774}. Our results prompt a reevaluation of certain findings where spectral invariance is assumed to hold.
\end{itemize}

These insights enhance our understanding of the complex scenario of the nonuniform dichotomy spectrum, offering new perspectives in both discrete and continuous settings. Contrary to what one might think, the nonuniform case is not a simple generalization of the uniform case and presents several challenges in obtaining spectral theorems.

\section{Notions of Dichotomy and Properties} 
    
In this section, we introduce the fundamental definitions that will be used throughout the paper. The primary objective is to present various notions of dichotomy and establish some of their key properties to characterize the behavior of solutions of nonautonomous equations. 

A \textbf{fundamental operator} for \eqref{700} is a map $X: \mathbb{Z} \to \mathcal{M}_d(\mathbb{R})$ that satisfies the operator equation 
$$X(k+1)=A(k)X(k),\qquad k\in \mathbb{Z},$$ and for every $k\in \mathbb{Z}$, there is an isomorphism between $\ker A(k)$ and $\ker A(k+1)$. On the other hand, the \textbf{evolution operator} or \textbf{transition matrix} for \eqref{700} is
 \begin{equation*}
    \Phi(k,n)= \left\{ \begin{array}{lc}
             A(k-1)A(k-2)\cdots A(n) &  \text{if}\,\,k > n, \\
             \\ I & \text{ if}\,\,k=n.
             \end{array}
   \right.
   \end{equation*}

Note that, in general, we do not assume that $A$ is nonsingular. However, if $A$ is nonsingular, the evolution operator for \eqref{700} can be expanded as $\Phi: \mathbb{Z} \times \mathbb{Z} \to GL_d(\mathbb{R})$, defined by
\begin{equation*}
    \Phi(k,n)= \left\{ \begin{array}{lc}
             A(k-1)A(k-2)\cdots A(n) &  \text{if}\,\,k > n, \\
             \\ I & \text{ if}\,\,k=n, \\
             \\ A^{-1}(k)A^{-1}(k+1)\cdots A^{-1}(n-1) & \text{if}\,\,k<n.
             \end{array}
   \right.
   \end{equation*}
Moreover, under this condition, given a fundamental operator $X$, the evolution operator can be expressed as $\Phi(k,n) = X(k)X(n)^{-1}$, since any fundamental operator in this case is also nonsingular. For dynamics in continuous time, analogous definitions are given to the concept of fundamental and evolution operators.

An \textbf{invariant projector} for \eqref{700} is a map $\P:\mathbb{Z}\to \mathcal{M}_d(\mathbb{R})$ of idempotents such that 
\begin{equation*}
    A(n) \P(n)=\P(n+1)A(n),
\end{equation*}
and the map
\[A(n)|_{\ker\P(n)}:\ker\P(n)\to \ker \P(n+1),\qquad \forall\,n\in \mathbb{Z},\]
is an isomorphism. Note that this implies each $\P(n)$ has the same rank. As a consequence, if $n\mapsto\P(n)$ is an invariant projector for \eqref{700}, then
\[\P(k)\Phi(k,n)=\Phi(k,n)\P(n),\qquad\forall\,n,k\in \mathbb{Z}.\]

If we fix an invariant projector for \eqref{700}, even when $n\mapsto A(n)$ may be singular, we can extend the definition of $\Phi$ to the case where $k < n$ via
\[\Phi(k,n)=\left[\Phi(n,k)|_{\ker \P(k)}\right]^{-1}:\ker \P(n)\to \ker \P(k).\]

The following definition is a crucial tool for describing behaviors more general than exponential ones. This concept is based on a similar notion presented in \cite{Silva discreto} for dynamics defined on $\mathbb{Z}^+$, and is also based on a similar concept presented in \cite{Silva} for continuous time flows. This notion has gained importance recently, as it has been employed on different works \cite{CJ,GJ}. Nonetheless, this concept has an already long trajectory in literature, as to the best of our knowledge, a preliminar notion of it first appeared in the works of Pinto and Naul\'in \cite{PN2,PN1,PN3} and has been applied to describe different scenarios \cite{Bento, ZFY}.

\begin{definition}
    We say that a map $\mu:\mathbb{Z}\to \mathbb{R}^+$ is a \textbf{discrete growth rate} if it is non-decreasing and verifies $\mu(0)=1$, $\lim_{n\to +\infty}\mu(n)=+\infty$ and $\lim_{n\to -\infty}\mu(n)=0$. 
\end{definition}

For instance, the map $n\mapsto e^n$ defines the \textbf{exponential growth rate}, while $p:\mathbb{Z}\to\mathbb{R}^+$ given by
   \begin{equation*}
		p(n):= \left\{ \begin{array}{lcc}
		n &  \text{ if } &   n> 0, \\
  1  &\text{if}& n=0,\\
		 \frac{1}{-n} & \text{ if }& n< 0,
		\end{array}
		\right.
		\end{equation*}
defines the \textbf{polynomial growth rate}. This last notion of growth has been employed on \cite{Dragicevic6} for dynamics defined on $\mathbb{Z}^+$ and a similar notion was used on \cite[Ex.~3.3]{CJ} in the context of continuous time dynamics. 

Other important examples that we will employ during the paper are: the map $\qeg:\mathbb{Z}\to \mathbb{R}^+$ given by 
   \begin{equation}\label{740}
		\qeg(n):= \left\{ \begin{array}{lcc}
		e^{n^2} &  \text{ if } &   n> 0, \\
  1  &\text{if}& n=0,\\
		 e^{-n^2} & \text{ if }& n< 0,
		\end{array}
		\right.
		\end{equation}
 which we call the \textbf{quadratic exponential growth rate}, and the map $\ceg:\mathbb{Z}\to \mathbb{R}^+$, given by $\ceg(n)=e^{n^3}$, which we call the \textbf{cubic exponential growth rate}. The cubic exponential growth rate was employed on \cite[Ex.~3.6]{CJ} in the context of continuous time dynamics.
 
For similar purposes, in \cite[p. 345]{Potzsche} the author introduces the concept of {\it generalized exponential function}, defined as a map $e_a(k, n)$ depending on two discrete variables. While this tool is more general, its primary application arises when $e_a(k, n) = \mu(k) / \mu(n)$ for some growth rate $\mu$.

In order to introduce the following definition, we employ the {\bf sign map} $\sgn:\mathbb{Z}\to \mathbb{R}$ given by $\sgn(n)=1$ if $n> 0$, $\sgn(0)=0$ and $\sgn(n)=-1$ if $n<0$.

\begin{definition}\label{DefNmuD}
    We say that system \eqref{700} admits a \textbf{nonuniform $\mu$-dichotomy} $(\NmuD)$ on $\mathbb{Z}$ if there is an invariant projector $n\mapsto \P(n)$ and constants $K\geq 1$, $\alpha<0$, $\beta>0$, $\nu,\theta\geq  0$, $\alpha+\theta<0$ and $\beta-\nu>0$ such that 
      \[\left\{ \begin{array}{lc}
            \norm{\Phi(k,n)\P(n)}\leq K\left(\dfrac{\mu(k)}{\mu(n)}\right)^\alpha \mu(n)^{\sgn(n)\theta}, &\forall\,\, k \geq n, \\
            \\ \norm{\Phi(k,n)[\Id-\P(n)]}\leq K\left(\dfrac{\mu(k)}{\mu(n)}\right)^\beta \mu(n)^{\sgn(n)\nu}, &\forall\,\, k\leq n.
             \end{array}
   \right.\]

   If $\theta,\nu=0$, we say it is a  \textbf{uniform $\mu$-dichotomy} $(\muD)$. If we do not impose the conditions $\alpha+\theta<0$ and $\beta-\nu>0$ we say it is a \textbf{slow nonuniform $\mu$-dichotomy} $(\sNmuD)$.
\end{definition}

In the cases of $\NmuD$ or $\sNmuD$, the above estimates are referred to as a dichotomy with \textbf{parameters} $(\P; \alpha, \beta, \theta, \nu)$ and constant $K$. For uniform dichotomies, we say that the previous estimates have parameters $(\P; \alpha, \beta)$. Note that if $\P = \Id$, then there is no parameter $\beta$ or $\nu$, while if $\P = 0$, there are no parameters $\alpha$ or $\theta$. In both cases, we retain the notation but replace the missing parameter with an asterisk $*$.

Note that every uniform dichotomy is, in particular, a nonuniform dichotomy, and every nonuniform dichotomy is, in particular, a slow nonuniform dichotomy.

\begin{remark}\label{slowproblem}
{\rm
In the context of continuous dynamics, the concept we refer to as {\it slow nonuniform dichotomy} has been employed in \cite{BV-CMP, BV-JDE, BV, HC,DZZ-PL,Huerta}. In these works, the authors address nonuniformity without requiring that the parameters associated with the nonuniformity (that is, $\theta$ and $\nu$) be smaller than the parameters related to convergence (that means, $\alpha,\beta$). Specifically, they do not impose the conditions $\alpha + \theta < 0$ and $\beta - \nu > 0$. However, in \cite{BV}, it is explicitly mentioned that the parameter associated with nonuniformity is {\it small enough}, but the exact meaning of this condition is not discussed. It is important to note that in these studies, the authors use the term {\it nonuniform dichotomy} to refer to this behavior, since the adjective {\it slow} was first coined in \cite{GJ}. On the other hand, some authors have explicitly considered these conditions, leading to what we term {\it nonuniform dichotomy}, as seen in \cite{Chu} for the exponential case and \cite{Silva} for the $\mu$-case. Additionally, some works have proposed a more stringent condition, defining them as the {\it exponential dichotomy with an arbitrarily small nonuniform part} \cite{BVD-AM,BV-RM, BVD-CCM}, among others.

In contrast, for discrete-time dynamics, the phenomenon we call slow nonuniform dichotomy appeared in \cite{DZZ}, while the concept of nonuniform dichotomy appeared in \cite{Silva discreto} for systems defined on $\mathbb{Z}^+$. It is worth noting that a concept similar to what we refer to as nonuniform dichotomy was considered in \cite{Chu2}, although in that work, the authors impose a stricter condition, equivalent to $\alpha + 2\theta < 0$ and $\beta - 2\nu > 0$.
}
\end{remark}

\begin{remark}\label{729}
 {\rm   In particular, for the exponential growth rate $\mu(n)=e^n$ we have that a nonuniform exponential dichotomy is given by
        \[\left\{ \begin{array}{lc}
            \norm{\Phi(k,n)\P(n)}\leq Ke^{\alpha(k-n)} e^{|n|\theta}, &\forall\,\, k \geq n, \\
            \\ \norm{\Phi(k,n)[\Id-\P(n)]}\leq Ke^{\beta(k-n)}e^{|n|\nu}, &\forall\,\, k\leq n,
             \end{array}
   \right.\]
while traditionally, nonuniform exponential dichotomies for discrete nonautonomous dynamics have been written as
\begin{equation}\label{728}
\left\{ \begin{array}{lc}
            \norm{\Phi(k,n)\P(n)}\leq Ka^{k-n}\varepsilon^{|n|}, &\forall\,\, k \geq n, \\
            \\ \norm{\Phi(k,n)[\Id-\P(n)]}\leq K(1/a)^{k-n}\varepsilon^{|n|}, &\forall\,\, k\leq n,
             \end{array}
   \right.
\end{equation}
for $a\in (0,1)$ and $\varepsilon\geq 1$. Nevertheless, it is straightforward to verify that by setting $a = e^\alpha$, $\beta = -\alpha$, $\varepsilon = e^\theta$, and $\theta = \nu$, both definitions become equivalent. We prefer the first formulation, as it enables us to generalize the construction to other growth rates. Furthermore, the notation used in this paper ensures a consistent unification of terminology for both difference and differential equations, aligning with the notation of recent works that investigate the spectral problem \cite{GJ,Silva}.
}
\end{remark}

\begin{remark}\label{730}
   {\rm  In \cite{Chu2}, the concept of nonuniform exponential dichotomy for difference equations is presented as a special case of Def.~\ref{DefNmuD}, where the growth rate is chosen as $\mu(n) = e^n$ for all $n \in \mathbb{Z}$, with constants $K \geq 1$, $\theta = \nu$, and $\alpha = -\beta$. Additionally, the condition $\alpha + 2\theta < 0$ is considered, which is equivalent to $a \varepsilon^2 < 1$, where $a = e^\alpha$ and $\varepsilon = e^\theta$ (using the multiplicative notation from \cite{Chu2}). This condition is strictly more stringent than $\alpha + \theta < 0$. Therefore, the set of possible values for $\alpha$ and $\theta$ is broader in Def.~\ref{DefNmuD}, as shown by the following inclusion:
   \[
\{(\alpha,\theta)\in(-\infty,0)\times[0,+\infty):\alpha+2\theta<0\}\subseteq\{(\alpha,\theta)\in(-\infty,0)\times[0,+\infty):\alpha+\theta<0\}.
   \]
   }
\end{remark}

\subsection{Properties of Each Notion of Dichotomy}

To characterize the behavior of solutions for each case, we begin by defining an invariant manifold associated with \eqref{700}. We say that $\mathcal{W} \subset \mathbb{Z} \times \mathbb{R}^d$ is an {\bf invariant manifold} for \eqref{700} if, for every $(n, \xi) \in \mathcal{W}$, it holds that $(k, \Phi(k, n) \xi) \in \mathcal{W}$ for all $k \in \mathbb{Z}$. In that case we write $\mathcal{W}(n)=\{\xi\in\mathbb{R}^d:(n,\xi)\in \mathcal{W}\}$ and call it the \textbf{fiber} over $n\in \mathbb{Z}$.

If $n\mapsto\P(n)$ is an invariant projector for \eqref{700}, then the sets
\[\im\,\P=\{(n,\xi):\xi\in  \im\,\P(n)\}\quad\text{ and }\quad\ker\P=\{(n,\xi):\xi\in \ker\P(n)\}\] 
are invariant manifolds for \eqref{700}. For system $\eqref{700}$, we define the sets
\[\mathcal{S}:=\left\{(n,\xi):\sup_{k\in \mathbb{Z}^+}\|\Phi(k,n)\xi\|<+\infty\right\}\]
and
\[\mathcal{U}:=\left\{(n,\xi):\sup_{k\in \mathbb{Z}^-}\|\Phi(k,n)\xi\|<+\infty\right\},\]
and call them the {\bf stable} and {\bf unstable} manifolds for \eqref{700}, respectively. It is easy to verify that they are indeed invariant manifolds for \eqref{700}.

Given two invariant manifolds $\mathcal{W}$ and $\mathcal{V}$ for \eqref{700}, we define their \textbf{intersection} by
\[\mathcal{W}\cap \mathcal{V}=\{(n,\xi)\in \mathbb{Z}\times\mathbb{R}^d:\xi\in\mathcal{W}(n)\cap \mathcal{V}(n) \},\]
and their \textbf{sum} by
\[\mathcal{W}+ \mathcal{V}=\{(n,\xi)\in \mathbb{Z}\times\mathbb{R}^d:\xi\in\mathcal{W}(n)+ \mathcal{V}(n) \}.\]
It is easy to verify that they are also invariant manifolds for \eqref{700}. If $\mathcal{W}\cap \mathcal{V}=\mathbb{Z}\times \{0\}$, we say their sum is a {\bf Whitney sum} and we write it as $\mathcal{W}\oplus\mathcal{V}$.

The following basic lemma establishes the first important relationship between the notions of dichotomy and the invariant manifolds associated with a system.

\begin{lemma}\label{704}
    Assume that system \eqref{700} has a dichotomy (either $\muD$, $\NmuD$ or $\sNmuD$) with invariant projector $n\mapsto\P(n)$. Then, $\im\, \P\subset \mathcal{S}$ and $\ker \P\subset \mathcal{U}$.
\end{lemma}
\begin{proof}
    Consider $(n,\xi)\in \im\,\P$, that is $\P(n)\xi=\xi$. From the dichotomy assumption, there are constants $K\geq 1$, $\alpha<0$ and $\theta\geq 0$ such that for every $k\geq n$:
    \[\|\Phi(k,n)\xi\|=\|\Phi(k,n)\P(n)\xi\|\leq K\left(\frac{\mu(k)}{\mu(n)}\right)^\alpha\mu(n)^{\sgn(n)\theta}|\xi|\leq K\mu(n)^{\sgn(n)\theta}|\xi|.\]

Therefore, as the set $\{k\in \mathbb{Z}^+:k<n\}$ is either empty or finite, we conclude $(n,\xi)\in \mathcal{S}$. The other contention is proved similarly.
\end{proof}

Now we present two important properties that are key to discern the differences between the three types of dichotomy.

\begin{definition}\label{769}
We say that system \eqref{700} has the
\begin{itemize}
    \item [(i)] \textbf{Unbounded Solutions Property} $(\USP)$ if it has no nontrivial bounded solutions,
    \item [(ii)]\textbf{Unique Projector Property} $(\UPP)$ if there exists at most one invariant projector for which it admits dichotomy (either $\muD$, $\NmuD$, or $\sNmuD$).
\end{itemize}
\end{definition}

\begin{remark}
    {\rm
We emphasize that the previous definition is not limited to difference systems; it also extends to nonautonomous linear differential equations in the context of continuous time ($t\in\mathbb{R}$).}
\end{remark}

The following lemma relates these properties to some notions of dichotomy.

\begin{lemma}\label{705}
    If the system \eqref{700} admits $\muD$ or $\NmuD$ with invariant projector $n\mapsto \P(n)$, then $\im\,\P=\mathcal{S}$ and $\ker\P=\mathcal{U}$, thus $\mathcal{S}\oplus\mathcal{U}=\mathbb{Z}\times \mathbb{R}^d$. Moreover, in this case the system \eqref{700} has both $\USP$ and $\UPP$.
\end{lemma}

\begin{proof}
Suppose the parameters in the dichotomy are $(\P;\alpha,\beta,\theta,\nu)$, where $\theta,\nu=0$ if it corresponds to a $\muD$. Choose $(n,\xi)\in \mathcal{S}$. By definition, there exist a constant $c>0$ such that
\begin{equation}\label{760}
    \|\Phi(k,n)\xi\|\leq c \quad\text{ for all }\,k\geq 0.
\end{equation}

    Consider $\xi=\xi_1+\xi_2$ with $\xi_1\in \im\,\P(n)$ and $\xi_2\in \ker\P(n)$. Then, for every $k\in\mathbb{Z}$ we have
    \begin{align*}
        \xi_2&=[\Id-\P(n)]\xi\\
        &=[\Id-\Phi(n,k)\Phi(k,n)\P(n)]\xi\\
        &=[\Id-\Phi(n,k)\P(k)\Phi(k,n)]\xi\\
        &=\Phi(n,k)[\Id-\P(k)]\Phi(k,n)\xi.
    \end{align*}

Thus, using \eqref{760} along with the dichotomy, for $k\geq 0$ we have
\[\|\xi_2\|\leq K\left(\frac{\mu(n)}{\mu(k)}\right)^\beta\mu(k)^\nu\|\Phi(k,n)\xi\|\leq K\,c\mu(k)^{-(\beta-\nu)}\mu(n)^{\beta},\qquad\forall k\geq n,\]
hence $\xi_2=0$, since $k$ can range free in $\mathbb{Z}^+$ and $\beta-\nu>0$. Therefore $\xi=\xi_1\in \im\,\P(n)$, which implies $(n,\xi)\in \im\,\P$. Now, as we already know $\im\,\P\subset\mathcal{S}$ from Lem. \ref{704}, we conclude $\im\,\P=\mathcal{S}$. 

We analogously prove $\ker\P=\mathcal{U}$. As these two identities uniquely define the projector, then the $\UPP$ is verified. Finally, from these identities we have $\mathcal{S}\cap\mathcal{U}=\im\,\P\cap\ker\P=\mathbb{Z}\cap \{0\}$. In other words, the only solution which is bounded in both $\mathbb{Z}^-$ and $\mathbb{Z}^+$ is the trivial one, hence the system has the $\USP$. 
\end{proof}

The previous lemma does not extend to the $\sNmuD$ case. Indeed, consider the following example:

\begin{example}\label{707}
    {\rm
    Consider the quadratic exponential growth rate $n\mapsto \qeg(n)$ defined on  \eqref{740}. Consider as well the scalar difference equation 
    $$x(n+1)=e^{-2n-1}x(n), \qquad\,n\in \mathbb{Z}.$$ 
    
    It is immediate to verify that its transition matrix is $\Phi(k,n)=e^{-(k^2-n^2)}$, $\forall \,k,n\in \mathbb{Z}$. In particular, the map $n\mapsto e^{-n^2}$ is a bounded nontrivial solution, thus this system does not have the $\USP$.

    On the other hand, for $k\geq n\geq 0$ we have
\begin{equation}\label{768}
    \Phi(k,n)=e^{-(k^2-n^2)}=\left(\frac{\qeg(k)}{\qeg(n)}\right)^{-1}\leq \left(\frac{\qeg(k)}{\qeg(n)}\right)^{-1}\qeg(n)^{\sgn(n)2},
\end{equation}
    while for $0\geq k\geq n$ we have $-2k^2+2n^2\leq 2n^2$, which implies 
    \[-(k^2-n^2)\leq -(n^2-k^2)+2n^2,\]
therefore
\[\Phi(k,n)=e^{-(k^2-n^2)}\leq e^{-(n^2-k^2)+2n^2}=\left(\frac{\qeg(k)}{\qeg(n)}\right)^{-1}\qeg(n)^{\sgn(n)2},\]
and finally, for $k\geq 0 \geq n$ we have
\begin{align*}
    \Phi(k,n)&=\Phi(k,0)\Phi(0,n)\\
    &\leq\left(\frac{\qeg(k)}{\qeg(0)}\right)^{-1}\qeg(0)^{\sgn(0)2}\left(\frac{\qeg(0)}{\qeg(n)}\right)^{-1}\qeg(n)^{\sgn(n)2}\\
    &=\left(\frac{\qeg(k)}{\qeg(n)}\right)^{-1}\qeg(n)^{\sgn(n)2}.
\end{align*}
    
In conclusion,  the following estimation is verified
\begin{equation}\label{763}
    \Phi(k,n)\leq \left(\frac{\qeg(k)}{\qeg(n)}\right)^{-1}\qeg(n)^{\sgn(n)2},\qquad \forall\,k\geq n,
\end{equation}
    hence the system admits $\sNmuD$ with parameters $(\Id;-1,*,2,*)$. In other words, the system has $\sNmuD$, but does not  present the $\USP$. Moreover, by a similar argument, it is verified that the system also admits the estimation
    \begin{equation}\label{764}
        \Phi(k,n)\leq \left(\frac{\qeg(k)}{\qeg(n)}\right)^{1}\qeg(n)^{\sgn(n)2},\qquad \forall\,k\leq n,
    \end{equation}
       thus the system also admits $\sNmuD$ with parameters $(0;*,1,*,2)$, which means that the $\UPP$ is not verified either, since there are two different projectors for which the system has $\sNmuD$.
    }
\end{example}

\begin{remark}
    {\rm
    The previous example has a differential counterpart on the equation
    \[\dot{x}=-2tx,\qquad t\in \mathbb{R},\] 
    since its evolution operator is given by $\Phi(t,s)=e^{-(t^2-s^2)}$, $\forall\,t,s\in \mathbb{R}$. Therefore, the previous estimations are also verified in this context. In other words, this differential equation also has $\sNmuD$ (with an analogous notion of quadratic exponential growth rate for continuous time), without verifying $\USP$ nor $\UPP$.
    }
\end{remark}

The following example illustrates that the lack of $\USP$ and $\UPP$ is an inherent issue with the definition of $\sNmuD$, and does not depend explicitly on any characteristic of the quadratic exponential growth rate, but is also present for the classic exponential rate.

\begin{example}\label{718}
{\rm
Consider the exponential growth rate $n\mapsto e^n$. Consider the map $A:\mathbb{Z}\to \mathbb{R}$, given by $A(n)=e$ for $n<-1$ and $A(n)=1/e$ for $n\geq -1$, and the equation $x(n+1)=A(n)x(n)$. A fundamental operator in this case is given by
   \begin{equation*}
		X(n)= \left\{ \begin{array}{lcc}
		e^{-n} &  \text{ if } &   n\geq 0, \\
		e^n & \text{ if }& n< 0.
		\end{array}
		\right.
		\end{equation*}

In particular, $X$ is a bounded nontrivial solution, therefore this system does not present the $\USP$. Denote the transition matrix for this system by $\Phi$. For $k\geq n\geq 0$ we have
\[\Phi(k,n)=e^{-(k-n)}\leq e^{-(k-n)+2|n|}, \]
while for $0\geq k\geq n$ we have $2k\leq 0$ and $|n|=-n$, which implies 
\[(k-n)\leq -(k-n)+2|n|.\]
Therefore
\[\Phi(k,n)=e^{k-n}\leq e^{-(k-n)+2|n|}.\]
Finally, for $k\geq 0\geq n$ we have
\begin{align*}
    \Phi(k,n)&=\Phi(k,0)\Phi(0,n)\\
    &\leq e^{-(k-0)+2|0|}e^{-(0-n)+2|n|}\\
    &=e^{-(k-n)+2|n|}.
\end{align*}

In conclusion, we deduce
\[\Phi(k,n)\leq e^{-(k-n)+2|n|},\qquad \forall\,k\geq n.\]
Thus, the system has $\sNmuD$ with parameters $(\Id;-1,*,2,*)$, while admitting a bounded nontrivial solution. Moreover, by an analogous computation we can conclude that for $k\leq n$ we have 
\[\Phi(k,n)\leq e^{k-n+2|n|},\]
thus the system has $\sNmuD$ with parameters $(0;*,1,*,2)$. Therefore, once again we have a system that exhibits $\sNmuD$ but does not verify either $\USP$ nor $\UPP$.
}
\end{example}

\begin{remark}
    {\rm
We can present a differential  system with a behavior similar to the previous example. Consider the map
   \begin{equation*}
		A(t)= \left\{ \begin{array}{lcc}
	-1&\text{ if }&t>\frac{\pi}{2}, \\
		-\sin(t) & \text{ if }& -\frac{\pi}{2}<t< \frac{\pi}{2},\\
        	1 &  \text{ if } &   t\leq -\frac{\pi}{2}.
		\end{array}
		\right.
		\end{equation*}

The reader can easily verify that by considering the continuous exponential growth rate, the equation $\dot{x}=A(t)x$ verifies $\sNmuD$, but does not have either $\USP$ nor $\UPP$.
    }
\end{remark}

As we can see, the notion of slow nonuniform dichotomy generally fails to exhibit the key characteristic that appropriately represent the notion of dichotomy as a consistent generalization of the autonomous concept of hyperbolicity, as the most common interpretation of this idea is the existence of a clear and uniquely defined splitting between stable and unstable solutions.

A sufficient condition for obtaining these two important properties, $\USP$ and $\UPP$, is to consider either the definitions of $\muD$ or $\NmuD$, as shown in Lem.~\ref{705}. However, in the next example we illustrate that this is far from a necessary condition.

\begin{example}\label{708}
    {\rm
Choose $\omega,a>0$ such that $3a>\omega>2a$ and consider the following difference equation on $\mathbb{Z}$
\[x(n+1)=\exp\left(-\omega+a(n+1)\cos(n+1)-an\cos(n)-a\sin(n+1)+a\sin(n)\right)x(n).\]

The transition matrix for this equation is given by
\[\Phi(k,n)=\exp\left(-\omega(k-n)+ak\cos(k)-an\cos(n)-a\sin(k)+a\sin(n)\right),\]
for all $k,n\in\mathbb{Z}$. If we consider the exponential growth rate, this system admits a $\sNED$ (slow nonuniform exponential dichotomy) with parameters $(\Id;-\omega+a,*,2a,*)$. In particular, the condition $\omega<3a$ implies that this $\sNED$ is not $\NED$. Nevertheless, it is a straightforward computation to verify that this system presents both the $\USP$ as every nontrivial solution is unbounded on $\mathbb{Z}^-$, and the $\UPP$, as it has dichotomy with identity projector, but it does not present dichotomy with projector zero (and these are the only options on dimension 1). 

We highlight that the previous system is the discrete analogous to the differential equation $\dot{x}=(-\omega-a t\sin t)x(t)$, for $t\in \mathbb{R}$, since it has the same evolution operator (but with continuous variables), which was introduced on \cite[Prop.~2.3]{BV} .
    }
\end{example}

\subsection{USPP Conjecture}

In light of the previous example, we propose a conjecture that links the $\USP$, $\UPP$, and the usual decomposition of invariant manifolds. So far, we have neither been able to prove it nor find a counterexample.

\begin{conjecture}\label{750}{\rm (USPP conjecture)}
    Assume that system \eqref{700} admits $\sNmuD$. Then, it exhibits the $\USP$ if and only if it exhibits the $\UPP$. Moreover, in this case $\im\, \P=\mathcal{S}$ and $\ker \P=\mathcal{U}$.
\end{conjecture}

As shown in Lem.~\ref{705}, this conjecture is verified when the slow dichotomy corresponds to either a $\muD$ or a $\NmuD$. Therefore, the interest of this proposed problem lies in studying its validity when there is a $\sNmuD$ that is not $\NmuD$.

As a first step in the direction of proving this conjecture, we present the following idea:

\begin{conjecture}\label{744}
    {\rm 
    If both $n\mapsto\P(n)$ and $n\mapsto\widetilde{\P}(n)$ are invariant projectors for \eqref{700} with which the system has $\sNmuD$, then they commute, {\it i.e.} $\P(n)\widetilde{\P}(n)=\widetilde{\P}(n)\P(n)$ for every $n\in \mathbb{Z}$. Moreover, in that case, one of the following holds:
    \begin{itemize}
        \item [(i)] for every $n\in \mathbb{N}$ there is some $\xi\in \mathbb{R}^d\setminus\{0\}$ such that $\P(n)\xi=\xi$ and $\widetilde{\P}(n)\xi=0$.
        \item [(ii)] $\P(n)=\widetilde{\P}(n)$ for every $n\in \mathbb{N}$.
    \end{itemize}
  
    }
\end{conjecture}

Once again, so far we have not been able to either prove or find a counterexample to this conjecture. If it was true, then we can consider the following lemma, which leads to a partial achievement of Conjecture \ref{750}

\begin{lemma}\label{703}
Assume Conjecture \ref{744} holds. If the system \eqref{700} admits $\sNmuD$ with parameters $(\P;\alpha,\beta,\theta,\nu)$, and if it has the $\USP$, then it also verifies the $\UPP$. Moreover, in this case $\im\, \P=\mathcal{S}$ and $\ker \P=\mathcal{U}$.
\end{lemma}

\begin{proof}
Suppose that system \eqref{700} has $\sNmuD$ but does not have the $\UPP$. This implies that it has $\sNmuD$ with two invariant projectors $n\mapsto \P(n)$ and $n\mapsto \widetilde{\P}(n)$ such that $\P\neq \widetilde{\P}$.  From Conjecture \ref{744}, for a fixed $n\in\mathbb{Z}$ there is some $\xi\in \mathbb{R}^d$ satisfying $\P(n)\xi=\xi$ and $\widetilde{\P}(n)\xi=0$, or in other words $ (n,\xi)\in  \im\P\cap \ker \widetilde{\P}$ with $\xi\neq 0$. By Lem.~\ref{704} we have $(n,\xi)\in \mathcal{S}\cap \mathcal{U}$. Therefore, the solution $k\mapsto \Phi(k,n)\xi$ is bounded in both $\mathbb{Z}^-$ and $\mathbb{Z}^+$, {\it i.e.} it is a nontrivial bounded solution, which contradicts the $\USP$. 
\end{proof}

\section{Spectrum and Spectral Theorems}

In this section, we present the notions of spectrum associated to the different types of dichotomy studied above. Our goal is to derive their respective spectral theorems. The first auxiliary concept we require is that of a weighted system, which serves as a tool for modifying systems in order to explore their spectra.

\begin{definition}
Let $\mu$ be a discrete growth rate. Given a real number $\gamma$, we define the \textbf{($\mu,\gamma$)-weighted} system (or simply $\gamma$-weighted system, if the growth rate is clear) associated to \eqref{700} as
\begin{equation}\label{701}
    x(k+1)=A(k)\left(\frac{\mu(k+1)}{\mu(k)}\right)^{-\gamma}x(k),\qquad k\in \mathbb{Z}.
\end{equation}
\end{definition}

Note that the evolution operator for the $\gamma$-weighted system is given by
\begin{equation}\label{761}
    \Phi_{\mu,\gamma}(k,n)=\left(\frac{\mu(k)}{\mu(n)}\right)^{-\gamma}\Phi(k,n),\qquad k,n\in \mathbb{Z},
\end{equation}
where $\Phi(k,n)$ is the evolution operator of the system \eqref{700}. If the growth rate $\mu$ is clear, we just denote it by $\Phi_\gamma$. Employing this concept we can define the different spectra of dichotomy.

\begin{definition}
    The \textbf{nonuniform $\mu$-dichotomy spectrum} of \eqref{700} is the set
    \[\Sigma_\NmuD(A):=\{\gamma\in \mathbb{R}: \eqref{701}\text{ does not admit }\NmuD\}.\]
   Moreover, its complement is the set $\rho_\NmuD(A)=\mathbb{R}\setminus \Sigma_\NmuD(A)$ called the \textbf{nonuniform $\mu$-resolvent set}. Analogously we define 
   \begin{itemize}
       \item [(i)] the \textbf{uniform $\mu$-dichotomy spectrum} $\Si_\muD(A)$,
       \item [(ii)] the \textbf{uniform $\mu$-resolvent set} $\rho_\muD(A)=\mathbb{R}\setminus \Si_\muD(A)$,
       \item [(iii)] the \textbf{slow nonuniform $\mu$-dichotomy spectrum} $\Si_\sNmuD(A)$, and
       \item [(iv)] the \textbf{slow nonuniform $\mu$-resolvent set} $\Si_\sNmuD(A)=\mathbb{R}\setminus\Si_\sNmuD(A)$.
   \end{itemize}  

\end{definition}

As every $\muD$ is in particular a $\NmuD$, which is itself a case of $\sNmuD$, we have the following inclusions:
\begin{equation}\label{741}
    \rho_\muD(A)\subset \rho_\NmuD(A)\subset \rho_\sNmuD(A).
\end{equation}
and consequently
\begin{equation}\label{742}
    \Sigma_\sNmuD(A)\subset \Sigma_\NmuD(A)\subset \Sigma_\muD(A).
\end{equation}

We call each connected component of $\rho$, for $\rho$ being either $\rho_\muD$ or $\rho_\NmuD$ or $\rho_\sNmuD$, a \textbf{spectral gap.} In the following we study the basic descriptions of these spectra. In the remainder of the section, unless stated otherwise, we fix a growth rate $\mu$. To begin, we prove that all three spectra are closed sets by showing all resolvent sets are open.

\begin{lemma}\label{706}
    The resolvent sets $\rho_\muD(A)$, $\rho_\NmuD(A)$ and $\rho_\sNmuD(A)$ are open.
\end{lemma}

\begin{proof}
Consider $\gamma,\zeta\in \mathbb{R}$. From \eqref{761}, we have the identity    $\Phi_\zeta(k,n)=\Phi_\gamma(k,n)\left(\frac{\mu(k)}{\mu(n)}\right)^{\gamma-\zeta}$. If $\gamma\in \rho_\muD(A)$, the $\gamma$-weighted system has $\muD$ with some parameters $(\P;\alpha,\beta)$. Then, we have estimations
      \[\left\{ \begin{array}{lc}
            \norm{\Phi_\gamma(k,n)\P(n)}\leq K\left(\dfrac{\mu(k)}{\mu(n)}\right)^\alpha , &\forall\,\, k \geq n, \\
            \\ \norm{\Phi_\gamma(k,n)[\Id-\P(n)]}\leq K\left(\dfrac{\mu(k)}{\mu(n)}\right)^\beta , &\forall\,\, k\leq n,
             \end{array}
   \right.\]
   thus
      \[\left\{ \begin{array}{lc}
            \norm{\Phi_\zeta(k,n)\P(n)}\leq K\left(\dfrac{\mu(k)}{\mu(n)}\right)^{\alpha+\gamma-\zeta} , &\forall\,\, k \geq n, \\
            \\ \norm{\Phi_\zeta(k,n)[\Id-\P(n)]}\leq K\left(\dfrac{\mu(k)}{\mu(n)}\right)^{\beta+\gamma-\zeta} , &\forall\,\, k\leq n.
             \end{array}
   \right.\]
therefore, for any $\zeta\in (\gamma-\epsilon,\gamma+\epsilon)$, where $\epsilon=\frac{1}{2}\min\{-\alpha,\beta\}$, we have a $\muD$ for the $\zeta$-weighted system with parameters $(\P;\alpha+\epsilon,\beta-\epsilon)$, or in other words, $\zeta\in \rho_\muD(A)$. This implies that for every $\gamma \in \rho_\muD(A)$, there exists a neighborhood of $\gamma$ that is also contained in $\rho_\muD(A)$, and hence, $\rho_\muD(A)$ is an open set.

Following a similar argument, if $\gamma \in \rho_\sNmuD(A)$ and the $\gamma$-weighted system has a $\sNmuD$ with parameters $(\P; \alpha, \beta, \theta, \nu)$, we can again choose $\zeta \in (\gamma - \epsilon, \gamma + \epsilon)$, where $\epsilon = \frac{1}{2} \min\{-\alpha, \beta\}$. This leads to $\zeta \in \rho_\sNmuD(A)$. Therefore, $\rho_\sNmuD(A)$ is also an open set.

Finally, for $\gamma \in \rho_\NmuD(A)$, we know that the $\gamma$-weighted system has a $\NmuD$ with parameters $(\P; \alpha, \beta, \theta, \nu)$. Therefore, if we choose $\zeta \in (\gamma - \widetilde{\epsilon}, \gamma + \widetilde{\epsilon})$, where $\widetilde{\epsilon} = \frac{1}{2} \min\{-(\alpha + \theta), \beta - \nu\}$, we conclude that $\zeta \in \rho_\NmuD(A)$. Thus, $\rho_\NmuD(A)$ is also an open set.
\end{proof}

\begin{remark}\label{746}
{\rm     
In the previous proof, in each case, whenever we choose $\zeta$ sufficiently close to $\gamma \in \rho(A)$, the projector corresponding to the dichotomy of the $\zeta$-weighted system is the same as the projector for the $\gamma$-weighted system. For both $\muD$ and $\NmuD$, which always exhibit the $\UPP$ (as shown in Lemma \ref{705}), this implies that the only possible projector for the $\zeta$-weighted system is the projector from the $\gamma$-weighted system. This is not the case for the $\sNmuD$ scenario. Even if the $\gamma$-weighted system with $\sNmuD$ may exhibit the $\UPP$ (as in Example \ref{708}), there is no guarantee that the $\zeta$-weighted system also satisfies the $\UPP$.
}
\end{remark}

We now turn to the task of establishing a criterion for when these spectra are bounded sets. To this end, consider the following definition.

\begin{definition}\label{NmuGrowth}
		The system \eqref{700} has \textbf{nonuniform $\mu$-bounded growth} with parameter $\epsilon>0$, or just \textbf{$(\Nmu,\epsilon)$-growth}, if there are constants $\widehat{K}\geq 1$, $a\geq 0$ such that
		\[
		\|\Phi(k,n)\|\leq \widehat{K}\left(\frac{\mu(k)}{\mu(n)}\right)^{\sgn(k-n)a}\mu(n)^{\sgn(n)\epsilon},\quad\forall\,k,n\in \mathbb{Z}.
		\]
		
		Moreover, if $\epsilon=0$, it is said that the system \eqref{700} has \textbf{uniform $\mu$-bounded growth} o just \textbf{$\mu$-growth}.
	\end{definition}

 \begin{remark}
     {\rm 
     In \cite{Chu2}, similarly as in Def.~\ref{NmuGrowth}, the authors consider the concept of {\it nonuniform exponential bound} for system \eqref{700}. This is defined by the existence of constants $K>0$, $\varepsilon\geq1$ and $\widetilde{a}\geq1$ such that 
     \[
     \|\Phi(k,n)\|\leq K\widetilde{a}^{|k-n|}\varepsilon^{|n|}, \qquad k,n\in\mathbb{Z}.
     \]
     Moreover, the parameters used in \cite{Chu2} to define the nonuniform exponential bound are the same as those used to describe the nonuniform exponential dichotomy, which may lead to confusion due to the overlap in notation. We emphasize that Def.~\ref{NmuGrowth} encompasses the concept of nonuniform exponential bound by simply considering the exponential growth rate $\mu(n)=e^{n}$, for all $n\in\mathbb{Z}$, together with constants $\widetilde{a}=e^{a}$ and $\varepsilon=e^{\epsilon}$. In addition, the parameters specified in Def.~\ref{NmuGrowth} are not necessarily the same with those in Def.~\ref{DefNmuD}, helping to clearly differentiate and avoid any possible misinterpretation.  
     }
 \end{remark}

\begin{lemma}\label{714}
    Assume that system \eqref{700} has $(\Nmu,\epsilon)$-growth (resp. $\mu$-growth). Then, both $\Sigma_\NmuD(A)$ and $\Sigma_\sNmuD(A)$ (resp. $\Sigma_\muD(A)$) are bounded sets.
\end{lemma}
\begin{proof}
    We prove it for the $\NmuD$ case and the others follow similarly. By hypothesis, there are constants $\widehat{K}\geq 1$, $a\geq 0$ such that
		\[
		\|\Phi(k,n)\|\leq \widehat{K}\left(\frac{\mu(k)}{\mu(n)}\right)^{\sgn(k-n)a}\mu(n)^{\sgn(n)\epsilon},\quad\forall\,k,n\in \mathbb{Z}.
		\]
        
Therefore, for every $\gamma>a+\epsilon$ and $k\geq n$ we have
  		\[
		\|\Phi_\gamma(k,n)\|\leq \widehat{K}\left(\frac{\mu(k)}{\mu(n)}\right)^{a-\gamma}\mu(n)^{\sgn(n)\epsilon},
		\]
  which corresponds to a $\NmuD$ with parameters $(\Id;a-\gamma,*,\epsilon,*)$. Thus, we obtain $(a+\epsilon,+\infty)\subset\rho_\NmuD(A)$. Similarly, for every $\gamma<-(a+\epsilon)$ and $k\leq n$ we have
    		\[
		\|\Phi_\gamma(k,n)\|\leq \widehat{K}\left(\frac{\mu(k)}{\mu(n)}\right)^{-\gamma-a}\mu(n)^{\sgn(n)\epsilon},
		\]
  which corresponds to a $\NmuD$ with parameters $(0;*,-\gamma-a,*,\epsilon)$. Thus, we obtain $(-\infty,-a-\epsilon)\subset\rho_\NmuD(A)$. Therefore we conclude $\Sigma_\NmuD(A)\subset [-a-\epsilon,a+\epsilon]$.
\end{proof}

For $\zeta,\gamma\in \mathbb{R}$, we can consider both the $\zeta$ and the $\gamma$-weighted systems, and study their stable and unstable manifolds, which we denote by $\mathcal{S}_\zeta$, $\mathcal{U}_\zeta$, $\mathcal{S}_\gamma$ and $\mathcal{U}_\gamma$, respectively. It is a straightforward computation to verify that if $\zeta>\gamma$, then
\begin{equation*}
    \mathcal{S}_\gamma\subset \mathcal{S}_\zeta\qquad \text{and}\qquad \mathcal{U}_\zeta\subset\mathcal{U}_\gamma.
\end{equation*}

In the following result we use the fact that for $\muD$ and $\NmuD$, the stable and unstable manifolds can be identified with the image and kernel of the invariant projector, as stated on Lemma \ref{705}. For this reason, we only stablish this result for $\rho_\muD$ and $\rho_\NmuD$. The case of $\rho_\sNmuD$ will be treated separately at the end of the section.

\begin{lemma}\label{753}
    Consider the system \eqref{700} and $\rho$ being either $\rho_\muD(A)$ or $\rho_\NmuD(A)$. If $J\subset \rho$ is an interval containing $\gamma$ and $\zeta$, then
    \begin{equation*}
    \mathcal{S}_\gamma= \mathcal{S}_\zeta\qquad \text{and}\qquad \mathcal{U}_\zeta=\mathcal{U}_\gamma.
\end{equation*}
\end{lemma}

\begin{proof}
    This is a consequence of the proof of Lemma \ref{706} and Remark \ref{746}. Explicitly, for $\muD$ and $\NmuD$ the projector is uniquely defined by the invariant manifolds and vice versa. As seen in the proof of Lemma \ref{706}, the projectors are constant on small enough neighborhoods on the resolvent set, therefore the stable and unstable invariant manifolds must be the same for all the elements lying on connected components.
\end{proof}

In the following we characterize the intersection between stable and unstable manifolds for different weighted systems.

\begin{lemma}\label{754}
Consider the system \eqref{700} and $\rho$ being either $\rho_\muD(A)$ or $\rho_\NmuD(A)$.  Consider $\zeta,\gamma\in \rho$ with $\zeta>\gamma$. Then, the following conditions are equivalent:
    \begin{itemize}
        \item [(i)] $\mathcal{S}_\zeta\cap \mathcal{U}_\gamma\neq \mathbb{R}\times \{0\}$,
        \item [(ii)] $[\gamma,\zeta]\cap \Sigma\neq \emptyset$, where $\Sigma$ is either $\Sigma_\muD(A)$ or $\Sigma_\NmuD(A)$, correspondingly
        \item [(iii)] $\rank\mathcal{S}_\gamma<\rank \mathcal{S}_\zeta$,
        \item [(iv)] $\rank \mathcal{U}_\gamma>\rank \mathcal{U}_\zeta$.
    \end{itemize}
\end{lemma}

The proof of the previous lemma follows the same ideas as the result \cite[Lemma 7]{Silva} and it is relatively standard in papers about dichotomy spectrum (see \cite{Chu2} for the exponential case, and \cite{Chu,Xiang} for the continuous time framework), thus we omit it and leave it as an exercise for the reader.

What we must understand from this result is that is establishes that every spectral gap, has a unique invariant projector associated to it, and no two different spectral gaps share the same projector.

\begin{remark}
    {\rm 
   The following is an immediate consequence of this result and the proof of Lemma \ref{714}: If the system \eqref{700} admits $(\Nmu,\epsilon)$-growth (resp. $\mu$-growth), then $\Sigma_\NmuD(A)$ (resp. $\Sigma_\muD(A)$) is a nonempty set. This assertion is supported by the fact that in the aforementioned proof we identified two different projectors for which the different weighted system have dichotomy, thus they must belong to different connected components, which implies that there is some element of the spectrum that separates them.
    }
\end{remark}

Employing lemmas \ref{706}, \ref{714}, \ref{753} and \ref{754}, and following once again a standard argument (see \cite{Chu2} for the exponential case, and \cite{Chu,Xiang} for the continuous time framework, specially \cite{Silva} since there the arguments are given for growth rates), we can prove the following results, known as the spectral theorems. For this reason, we omit the proof it and leave it as an exercise for the reader.

\begin{theorem}\label{770}
    (Spectral Theorem for ${\NmuD}$) Consider the nonautonomous difference equation \eqref{700}. There exists some $m\in \{0,1,\dots,d\}$ such that
		\[
		\Si_\NmuD(A)=\bigcup_{i=1}^m[a_i,b_i],
		\]
		for some $a_i,b_i\in \mathbb{R}\cup \{\pm \infty\}$ verifying $a_i\leq b_i<a_{i+1}$. The following are also verified:
\begin{itemize}
    \item [(i)] For each $\gamma\in \rho_{\NmuD}(A)$, the $\gamma$-weighted system has $\NmuD$ with a unique invariant projector $n\mapsto \P(n)$. 
    \item [ii)] For $\zeta,\gamma\in \rho_{\NmuD}(A)$, the projectors of their respective weighted systems are the same if and only if they lie in the same spectral gap.  
    \item [(iii)] If $\zeta,\gamma\in \rho_\NmuD(A)$ lie on different spectral gaps and $\zeta>\gamma$, then $\mathcal{U}_\zeta\subsetneq \mathcal{U}_\gamma$ and $\mathcal{S}_\gamma\subsetneq \mathcal{S}_\zeta$.
    \item [(iv)] If the system \eqref{700} has $(\Nmu,\epsilon)$-growth, the spectrum is nonempty and bounded, that is $m\neq 0$ and $a_1\neq-\infty$ and $b_m\neq +\infty$. 
\end{itemize}
\end{theorem}

An analogous theorem is established for the $\muD$ case.

\subsection{The Slow Nonuniform Case}

To conclude this section, we introduce a novel notion of the spectrum. As we verified in the previous subsection, the $\sNmuD$ behavior does not, in general, provide the appropriate splitting between stable and unstable manifolds. On the other hand, the $\NmuD$ behavior does indeed ensure this splitting, but it is strictly more stringent. In terms of the involved spectra, this means that, in general, the $\NmuD$ spectrum is larger than the $\sNmuD$ spectrum. We emphasize that the goal is typically to have a spectrum that is as small as possible, as this facilitates the achievement of certain conditions, such as spectral nonresonance (see \cite{Song} for the discrete case, and \cite{CJ} for the continuous case), which is important for linearization purposes. In an effort to find a middle ground, we present the following concept.

\begin{definition}\label{771}
    We define the set 
    \[\rho_\sNmuD^\UPP(A):=\{\gamma\in \mathbb{R}: \eqref{701} \text{ admits }\sNmuD\text{ having the }\UPP\}\]
    and call it the \textbf{unique projector slow nonuniform $\mu$-resolvent set}. Its compliment $\Sigma_\sNmuD^\UPP(A):=\mathbb{R}\setminus\rho_\sNmuD^\UPP(A)$ is called the \textbf{unique projector slow nonuniform $\mu$-dichotomy spectrum}. 
\end{definition}

From Lemma \ref{705}, we immediately obtain the following chain of inclusions, updating the forms presented in \eqref{741} and \eqref{742}.
\[\rho_\muD(A)\subset \rho_\NmuD(A)\subset \rho_\sNmuD^{\UPP}(A)\subset \rho_\sNmuD(A),\]
\begin{equation}\label{767}
    \Sigma_\sNmuD(A)\subset \Sigma_\sNmuD^\UPP(A)\subset \Sigma_\NmuD(A)\subset \Sigma_\muD(A).
\end{equation}

An immediate consequence from this observation and Lemma \ref{714} is that if a system has nonuniform bounded growth, then $\Sigma_\sNmuD^\UPP(A)$ is a bounded set.

    We also call each connected component of $\rho_\sNmuD^\UPP$ a spectral gap. Note that each spectral gap of $\rho_\muD$ is contained in a unique spectral gap of $\rho_\NmuD$, and each spectral gap of $\rho_\NmuD$ is contained in a unique spectral gap of $\rho_\sNmuD^\UPP$. Nevertheless, it can happen that $\rho_\sNmuD^\UPP$ has more spectral gaps than $\rho_\NmuD$, which can have more spectral gaps than $\rho_\muD$.

\begin{example}\label{799}
{\rm
    Considering the quadratic exponential growth rate and the system on Example \ref{707}. From \eqref{763}, we have
    \begin{equation}\label{765}
            \Phi_\gamma(k,n)\leq \left(\frac{\qeg(k)}{\qeg(n)}\right)^{-(1+\gamma)}\qeg(n)^{\sgn(n)2},\qquad \forall\,k\geq n,
    \end{equation}
which implies that the $\gamma$-weighted system has \textrm{sNqD} with parameters $(\Id;-(1+\gamma),*,2,*)$ for all $\gamma>-1$. Therefore, we have $(-1,+\infty)\subset \rho_{\textrm{sNqD}}$. Similarly, from \eqref{764}  we obtain 
\begin{equation}\label{766}
    \Phi_\gamma(k,n)\leq \left(\frac{\qeg(k)}{\qeg(n)}\right)^{1-\gamma}\qeg(n)^{\sgn(n)2},\qquad \forall\,k\leq n,
\end{equation}
hence $\gamma$-weighted system has \textrm{sNqD} with parameters $(0;*,1-\gamma,*,2)$ for all $\gamma<1$. Thus, we have $(-\infty,1)\subset \rho_{\textrm{sNqD}}$. In conclusion, $\Sigma_{\textrm{sNqD}}=\emptyset$.

Note that for $\gamma\in (-1,1)$, the $\gamma$-weighted system has \textrm{sNqD} with both projectors $\Id$ and $0$. Therefore, for $\gamma\in(-1,1)$, the $\gamma$-weighted system does not have the $\UPP$. In other words, $(-1,1)\subset \Sigma_{\textrm{sNqD}}^\UPP$.

Once again from \eqref{765}, we deduce $(1,\infty)\subset \rho_{\textrm{NqD}}$, while from \eqref{766} we obtain $(-\infty,-1)\subset \rho_{\textrm{NqD}}$. Thus $\Sigma_{{\textrm{NqD}}}\subset [-1,1]$. Following these arguments and the inclusions described on \eqref{767} we have
\[
(-1,1)\subset\Sigma_{\textrm{sNqD}}^\UPP\subset \Si_{{\textrm{NqD}}}\subset[-1,1],
\]
and as $\Sigma_{{\textrm{NqD}}}$ is closed by Lemma \ref{706}, we have $\Sigma_{{\textrm{NqD}}}=[-1,1]$. 

We already know that the $(-1)$-weighted system has \textrm{sNqD}. It remains to be seen that the $(-1)$-weighted system has the $\UPP$. As the system is one dimensional,  it is enough to prove that it does not have \textrm{sNqD} with the only other available projector, which is the $\Id$ projector. From \eqref{768} we know
\[
\Phi_{-1}(k,n)=e^{-(k^2-n^2)}\left(\frac{\qeg(k)}{\qeg(n)}\right)^{-(-1)}=1,\qquad\forall \,k\geq n\geq 0,
\]
hence $\lim_{k\to \infty}\Phi_{-1}(k,n)\neq 0$. This implies that the $(-1)$-weighted system does not have dichotomy with the projector $\Id$, thus its dichotomy with projector $0$ does indeed present the $\UPP$. In other words, we have $-1\in\rho_{\textrm{sNqD}}^\UPP$. Similarly we prove $1\in \rho_{\textrm{sNqD}}^\UPP$.

In conclusion,
    \[\Sigma_\textrm{sNqD}=\emptyset\subset\Sigma_\textrm{sNqD}^\UPP= (-1,1)\subset\Sigma_{\textrm{NqD}}=[-1,1]\subset \Sigma_\textrm{qD}=\mathbb{R}.\]
}
\end{example}

\begin{example}\label{798}
 {\rm      Consider now the exponential growth rate and the system on Example \ref{708}. We have that $0$ does not belong to the $\Sigma_{\sNED}$ or $\Sigma_{\sNED}^\UPP$ spectra, since this system has a $\sNED$ which does exhibit $\UPP$, but $0$ does belong to its $\Sigma_\NED$ spectrum, since this $\sNED$ is not a $\NED$. Explicitly, for this example we have
    \[\Sigma_\sNED=\Sigma_\sNED^\UPP=[-\omega-a,-\omega+a]\subset \Sigma_{\NED}=[-\omega-3a,-\omega+3a]\subset \Sigma_\ED=\mathbb{R}.\]

    It it worth noting that the spectra for this example were carefully computed on \cite{GJ} on the continuous case and we compute it explicitly on Example \ref{773}.
    }
\end{example}

\begin{remark}
    {\rm 
    As seen by the previous examples, all four spectra can be different.
    }
\end{remark}

Now we delve into the task of finding properties for this spectrum.

\begin{definition}
    Consider the system \eqref{700}. The \textbf{dimension map} $\dim\colon\rho_\sNmuD^\UPP(A)\to \mathbb{Z}$ is defined by $\gamma\mapsto \rank\P$,  where $n\mapsto\P(n)$ is the invariant projector for the $\sNmuD$ of the $\gamma$-weighted system.
\end{definition}

Note that this map is well defined, since by definition, for $\gamma\in \rho_\sNmuD^\UPP(A)$ there is one and only one projector with which the respective $\gamma$-weighted system admits $\sNmuD$. However, note that this map cannot be extended to $\rho_\sNmuD(A)$, since there are systems having multiple $\sNmuD$, with projectors which have different rank, as stated in Examples \ref{707}  and \ref{718}.

\begin{lemma}
 Consider the system \eqref{700}. The dimension map is locally constant on $\rho_\sNmuD^\UPP(A)$. Thus, it is constant on spectral gaps.
\end{lemma}

\begin{proof}
    Choose $\gamma\in \rho_\sNmuD^\UPP(A)$ and fix that the unique invariant projector associated to the $\gamma$-weighted system is $n\mapsto\P(n)$. From Lemma \ref{706}, there exists $\epsilon>0$ such that for every $\zeta\in (\gamma-\epsilon,\gamma+\epsilon)$, the $\zeta$-weighted system admits $\sNmuD$ with projector $n\mapsto\P(n)$. Now, $(\gamma-\epsilon,\gamma+\epsilon)\cap\rho_\sNmuD^\UPP(A)$ is a neighborhood of $\gamma$ in $\rho_\sNmuD^\UPP(A)$, and for every $\zeta$ in this intersection the projector is unique and is $n\mapsto\P(n)$.
\end{proof}

In other words, we have seen that every spectral gap in $\rho_\sNmuD^\UPP(A)$ has a unique associated invariant projector. In the following we will study the reciprocal, that is, if two elements of the resolvent have the same dimension,   then they are in the same spectral gap.

\begin{remark}
    {\rm
The dimension map can be restricted to both $\rho_{\NmuD}(A)$ and $\rho_\muD(A)$, since these are contained in $\rho_\sNmuD^\UPP(A)$. Moreover, each spectral gap of either $\rho_{\NmuD}(A)$ or $\rho_\muD(A)$ is contained in a unique spectral gap of $\rho_\sNmuD^\UPP(A)$. Nevertheless, it can happen that $\rho_\sNmuD^\UPP(A)$ has more spectral gaps that the other two resolvent sets.
    }
\end{remark}

\subsection{Consequences of USPP Conjecture}
In the remainder of this section, we assume the validity of the USPP conjecture (Conjecture \ref{750}).

\begin{lemma}\label{713}
Assume $\USPP$ conjecture holds. Consider $\zeta,\gamma\in \rho_\sNmuD^\UPP(A)$ such that $\dim(\gamma)=\dim(\zeta)$. Then, the projector is the same for the $\gamma$ and $\zeta$-weighted systems.
\end{lemma}

\begin{proof}
    Assume without loss of generality $\zeta>\gamma$. Consider that the projector for the $\gamma$-weighted system is $n\mapsto \P(n)$ and the projector for the $\zeta$-weighted system is $n\mapsto \Q(n)$. Note that for $k\geq n$ we have
    \[\|\Phi_\zeta(k,n)\P(n)\|\leq \|\Phi_\gamma(k,n)\P(n)\|,\]
    thus, all the solutions of the $\zeta$-weighted system which lie in $\im \,\P$ are bounded on $\mathbb{Z}^+$. But as the projector of the $\zeta$-weighted system is $\Q$, and $\USPP$ conjecture holds, we know $\mathcal{S}_\zeta=\im \,\Q$. Therefore, $\im \,\P\subset \im\, \Q$, and as they have the same dimension, then $\im\, \P=\im\,\Q$. Analogously we prove $\ker \P=\ker\Q$, thus they must be the same projector. 
\end{proof}

\begin{lemma}\label{711}
Assume $\USPP$ conjecture holds. Consider $\zeta,\gamma\in \rho_\sNmuD^\UPP(A)$ such that the projector is the same for the $\gamma$ and $\zeta$-weighted systems. Then, they belong to the same spectral gap $\rho_\sNmuD^\UPP(A)$. 
\end{lemma}

\begin{proof}
    Suppose $\zeta>\gamma$. We will prove that for any $\eta\in (\gamma,\zeta)$, the $\eta$-weighted system has dichotomy with the same projector as the $\zeta$ and $\gamma$-weighted systems. Set that the $\gamma$-weighted system has $\sNmuD$ verifying $\USP$ with parameters $(\P;\alpha,\beta,\theta,\nu)$, while the $\zeta$-weighted system has $\sNmuD$ verifying $\USP$ with parameters $(\P;\walpha,\wbeta,\wtheta,\wnu)$. For $k\geq n$ we have
    \begin{equation}\label{715}
        \|\Phi_\zeta(k,n)\|\leq \|\Phi_\eta(k,n)\|\leq\|\Phi_\gamma(k,n)\|,
    \end{equation}
therefore
\begin{equation}\label{710}
    \|\Phi_\eta(k,n)\P(n)\|\leq\|\Phi_\gamma(k,n)\P(n)\|\leq K\left(\dfrac{\mu(k)}{\mu(n)}\right)^\alpha \mu(n)^{\sgn(n)\theta},
\end{equation}
   for some $K\geq 1$. On the other hand, for $k\leq n$ we have
\begin{equation}\label{716}
    \|\Phi_\gamma(k,n)\|\leq \|\Phi_\eta(k,n)\|\leq\|\Phi_\zeta(k,n)\|,
\end{equation}
  therefore
  \[\|\Phi_\eta(k,n)[\Id-\P(n)]\|\leq\|\Phi_\zeta(k,n)[\Id-\P(n)]\|\leq  \widetilde{K}\left(\dfrac{\mu(k)}{\mu(n)}\right)^\wbeta \mu(n)^{\sgn(n)\wnu},\]
 for some $\widetilde{K}$. The previous estimation, in conjunction to \eqref{710}, implies that the $\eta$-weighted system has $\sNmuD$ with parameters $(\P;\alpha,\wbeta,\theta,\wnu)$.

It remains to prove that the $\eta$-weighted system has the $\UPP$. As we assume the $\USPP$ conjecture is true, it is enough to verify it has the $\USP$.

Suppose a solution $k\mapsto \Phi_\eta(k,n)\xi$ is bounded on $\mathbb{Z}$. By \eqref{715}, the solution to the $\zeta$-weighted system $k\mapsto \Phi_\zeta(k,n)\xi$ is bounded on $\mathbb{Z}^+$.

Thus, as the $\zeta$-weighted system has the $\USP$ (because it has the $\UPP$, and we suppose the $\USPP$ conjecture holds), we have by Lemma \ref{703} that $(n,\xi)\in \im\,\P$. Similarly, from \eqref{716} we have that the solution of the $\gamma$-weighted system $k\mapsto \Phi_\gamma(k,n)\xi$ is bounded on $\mathbb{Z}^-$, therefore, as the $\gamma$-weighted system has the $\USP$, by Lemma \ref{703} we have $(n,\xi)\in \ker \P$. As $\im\,\P\cap \ker \P=\mathbb{Z}\times \{0\}$, we have $\xi=0$. In other words, if a solution to the $\eta$-weighted system is bounded, it must be trivial, {\it i.e.} it has the $\USP$. In summary, $\eta\in \rho_\sNmuD^\UPP(A)$.

As $\eta$ was arbitrary, then $(\gamma,\zeta)\subset \rho_{\NmuD}^\UPP$, which implies the claim.
\end{proof}

\begin{theorem}\label{772}
    (Spectral Theorem for ${\sNmuD}$ verifying $\UPP$) Assume $\USPP$ conjecture holds. Consider the nonautonomous difference equation \eqref{700}. There exists some $m\in\{0,1,\ldots,d\}$ such that 

    \[
    \Sigma_\sNmuD^\UPP(A)=\bigcup_{i=1}^{m} I_i,
    \]
    where $I_i$ are non-overlapping intervals (maybe empty). Moreover, the following hold: 
    \begin{itemize}
        \item[(i)] If the system \eqref{700} has $(\Nmu,\epsilon)$-growth, then $m\neq 0$, {\it i.e.} the spectrum is nonempty, and bounded.
        \item[(ii)] For every $\gamma\in \rho_{\sNmuD}^\UPP(A)$, the $\gamma$-weighted system has a $\sNmuD$ with a unique invariant projector $n\mapsto \P(n)$.
        \item[(iii)] For $\gamma,\zeta\in \rho_{\sNmuD}^\UPP(A)$, the projectors of their respective weighted systems are the same if and only if $\gamma$ and $\zeta$ belong to the same spectral gap. 
    \end{itemize}
\end{theorem}

\begin{proof}
By Lemmas \ref{713} and \ref{711}, every connected component of $\rho_\sNmuD^\UPP(A)$ has a unique dimension of projector associated to it. Now note that the dimensions allowed are integers bounded between zero and the dimension of the space. Therefore, the resolvent $\rho_\sNmuD^\UPP(A)$ has finite connected subsets. 

As stated at the beginning of this subsection, if the system has $(\Nmu,\epsilon)$-growth, then $\Sigma_\sNmuD^\UPP(A)$ is once again bounded. From the proof of this Lemma, we see that the projector of the $\gamma$-weighted system is $\Id$ for $\gamma$ large enough and $0$ for $\gamma$ small enough. Thus, as the projector is not the same, this implies that there exist at least two different connected subsets of $\rho_{\sNmuD}^\UPP(A)$, therefore $\Sigma_\sNmuD^\UPP(A)$ is not empty. 
\end{proof}

\begin{remark}
    {\rm    As we can see, we do not have a conclusion stating that the $\Sigma_\sNmuD^\UPP$ spectrum is closed, as it is the case for the $\muD$, $\NmuD$ or $\sNmuD$ spectra. We would like to emphasize that in general the intervals $I_i$ involved in Theorem \ref{772} can be open, closed or even mixed, as seen on Examples \ref{799} and \ref{798}.
    }
\end{remark}

\begin{remark}
{\rm    It remains to be seen if the $\Sigma_\sNmuD$ spectrum also has a finite number of connected components (we already know it is always closed, and if it exhibits nonuniform bounded growth then it is also compact and nonempty). Nevertheless, the whole previous conclusion does not extend to this kind of dichotomy, since we know projectors are not unique in this case. In other words, the spectral gaps for the $\sNmuD$ case do not have a unique projector, therefore we cannot in general stablish a correspondence between projectors and spectral gaps. 
}
\end{remark}

\section{Optimal Dichotomy Constants}\label{optdichconst}

Assume that $\mu:\mathbb{Z}\to \mathbb{R}^+$ is a discrete growth rate. If system \eqref{700} exhibits any type of dichotomy, the parameters involved are not unique. This means the system can also admit a dichotomy under slight perturbations of the initial parameters for which it exhibits a dichotomy. For instance, if system \eqref{700} admits $\sNmuD$ with parameters $(\P;\alpha,\beta,\theta,\nu)$, then it immediately admits $\sNmuD$ with parameters $(\P;\walpha,\wbeta,\wtheta,\wnu)$ for any $\walpha\in(\alpha,0)$, $\wbeta\in (0,\beta)$, $\wtheta\geq \theta$ and $\wnu\geq \nu$. Analogously, if \eqref{700} admits $\NmuD$ with parameters $(\P;\alpha,\theta,\beta,\nu)$, then it also admits $\NmuD$ with parameters  $(\P;\walpha,\wbeta,\wtheta,\wnu)$ for any $\widetilde{\alpha}\in (\alpha,-\theta)$, $\widetilde{\beta}\in (\nu,\beta)$, $\widetilde{\theta}\in (\theta,-\widetilde{\alpha})$, and $\widetilde{\nu}\in (\nu,\widetilde{\beta})$. 

The aforementioned behavior of the dichotomy parameters motivates the following definition; see also \cite[Def.~2.5]{GJ}.

	\begin{definition}
Assume the system \eqref{700} admits $\NmuD$ with an invariant projector $\P$.   The \textbf{region of stable constants} for \eqref{700} is the set defined by
		\[
		\St_\P:=\{(\alpha,\theta)\in \mathbb{R}^2: \text{ \eqref{700} admits }\NmuD\text{ with parameters }(\P;\alpha,\beta,\theta,\nu), \text{ for some }\beta,\nu\},
		\]
		and the \textbf{region of unstable constants} for \eqref{700}  is the set defined by    \[
		\Un_\P:=\{(\beta,\nu)\in \mathbb{R}^2: \text{ \eqref{700} admits }\NmuD\text{ with parameters }(\P;\alpha,\beta,\theta,\nu), \text{ for some }\alpha,\theta\}.
		\]
	\end{definition}

The concept of these regions is also applicable considering the notions of $\muD$ or $\sNmuD$ instead of $\NmuD$. However, as we will see in Prop.~\ref{boundedregions}, the case of $\NmuD$ is particularly interesting if the notion of bounded growth for system \eqref{700} is satisfied.

\begin{remark}\rm 
    In the next result, we will prove item (ii). It is worth noting that Gallegos and Jara, in \cite[Lem.~2.10]{GJ}, only proved (i), leaving (ii) as an exercise for the reader, as its proof is quite similar. Nevertheless, since this result is recent, we will provide an explicit proof of (ii) to offer the reader a complete version of the proposition. Similarly, in this article we will present the part of the proof that is omitted in \cite[Lem.~2.17; Prop.~2.18 \& Lem.~2.19]{GJ} to ensure a detailed and comprehensive understanding of these results; see Prop.~\ref{743} and Thm.~\ref{limits}.  

    It is worth noting that \cite{GJ} is developed on the continuous time framework, but the proving techniques can be directly replicated.
\end{remark}

\begin{proposition}\label{boundedregions}
    Assume the system \eqref{700} admits $\NmuD$ with parameters {\rm ($\P$;$\alpha$,$\beta$,$\theta$,$\nu$)} and {\rm($\Nmu$,$\epsilon$)}-growth with constants $a\geq0$, $\widehat{K}\geq 1$. Then the following inclusions hold 
    \begin{itemize}
        \item[(i)] $\St_{\P}\subset\{(\alpha,\theta)\in(-\infty,0)\times[0,+\infty):-(a+\epsilon)\leq\alpha<\theta< a+\epsilon\}$,  whenever $\P\neq0$.
        \vspace{0.1cm}
        \item[(ii)] $\Un_{\P}\subset\{(\beta,\nu)\in(0,+\infty)\times[0,+\infty):0\leq\nu<\beta\leq a+\epsilon\}$,  whenever $\P\neq\Id$.
    \end{itemize}
\end{proposition}
\begin{proof}
    Let us demonstrate (ii). Assume $k<n$ and $\P\neq\Id$. From above, we can choose $\xi\in(\Id-\P(k))\mathbb{R}^{d}$ with $\xi\neq0$. Then we deduce 
    \begin{equation*}
        \begin{split}
            \|\xi\|&=\|\xi-\P(k)\xi\|=\|\Phi(k,n)\{\Id-\P(n)\}\Phi(n,k)\xi\|\\
            &\leq K\left(\dfrac{\mu(k)}{\mu(n)}\right)^{\beta}\mu(n)^{\sgn(n)\nu}\widehat{K}\left(\dfrac{\mu(n)}{\mu(k)}\right)^{\sgn(n-k)a}\mu(k)^{\sgn(k)\epsilon}\|\xi\|\\
            &=K\widehat{K}\mu(k)^{\beta-a+\sgn(k)\epsilon}\mu(n)^{a+\sgn(n)\nu-\beta}\|\xi\|,
        \end{split}
    \end{equation*}
    which in turn implies 
    \begin{equation}\label{719}   
    1\leq K\widehat{K}\mu(k)^{\beta-a+\sgn(k)\epsilon}\mu(n)^{a+\sgn(n)\nu-\beta}.
    \end{equation}
    Now, considering $\beta-a-\epsilon>0$, and taking $k\to-\infty$, the right-hand side of \eqref{719} tends to zero, which is a contradiction. Hence, we must have $\beta\leq a+\epsilon$. The first inclusion (i) follows by using a similar argument.  
\end{proof}

\begin{remark}
    {\rm
    From Prop.~\ref{boundedregions}, we conclude that if system \eqref{700} exhibits $\NmuD$ and ($\Nmu$,$\epsilon$)-growth, then both the stable region $\St_{\P}$ and the unstable region $\Un_{\P}$ are bounded sets. On the other hand, if the system \eqref{700} admits $\muD$ and $\mu$-growth then $\St_{\P}\subset [-a,0)\times\{0\}$ and $\Un_{\P}\subset(0,a]\times\{0\}$. Therefore, both sets are bounded. However,  if we consider the system \eqref{700} admitting $\sNmuD$ and satisfying ($\Nmu$,$\epsilon$)-growth, we cannot guarantee the boundedness of the regions $\St_{\P}$ and $\Un_{\P}$. In the case of the stable region $\St_{\P}$, while the parameter $\alpha$ is restricted to the range $[-a - \epsilon, 0)$, the parameter $\theta$ possesses no restriction apart from $\theta \geq 0$. Similarly, for the unstable region $\Un_{\P}$, we find that $0 < \beta \leq a + \epsilon$, while $\nu$ can vary freely within the range $[0, +\infty)$.
    }
\end{remark}

We are in a position to introduce the concept of optimal parameters associated to a system admitting $\NmuD$; see \cite[Def.~2.12]{GJ} 
\begin{definition}
		Assume the system \eqref{700} admits  $(\Nmu,\epsilon)$-growth and $\NmuD$ with an invariant projector $\P$. 
		\begin{itemize}
			\item [(i)] The \textbf{optimal stable ratio} is defined by
			\[
			\st_\P:=\inf\{\alpha+\theta:(\alpha,\theta)\in \St_\P\}.
			\]
			\item [(ii)] The \textbf{optimal unstable ratio} is defined by
			\[
			\un_\P:=\sup\{\beta-\nu:(\beta,\nu)\in \Un_\P\}.
			\]
		\end{itemize}
	\end{definition}

\begin{lemma}
    Assume the system \eqref{700} admits $\NmuD$ and {\rm ($\Nmu$,$\epsilon$)}-growth with constants $\widehat{K}\geq1$ and $a>0$. Then the $\gamma$-weighted system \eqref{701} has {\rm($\Nmu$,$\epsilon$)}-growth with constants $\widehat{K}$ and $a+|\gamma|$. 
\end{lemma}
\begin{proof}
    The assertion follows from the next estimation
    \begin{equation*}
        \begin{split}
            \|\Phi_{\gamma}(k,n)\|&=\left(\dfrac{\mu(k)}{\mu(n)}\right)^{-\gamma}\|\Phi(k,n)\|
            \leq \widehat{K}\left(\dfrac{\mu(k)}{\mu(n)}\right)^{\sgn(k-n)a - \gamma} \mu(n)^{\sgn(n)\epsilon}\\
            &\leq\widehat{K}\left(\dfrac{\mu(k)}{\mu(n)}\right)^{\sgn(k-n)\{a + |\gamma|\}}\mu(n)^{\sgn(n)\epsilon}, \quad \text{for $k,n\in\mathbb{Z}$}.
        \end{split}
    \end{equation*}
\end{proof}

Assume the system \eqref{700} admits {\rm ($\Nmu$,$\epsilon$)}-growth with constants $\widehat{K}\geq1$ and $a>0$. Then, it is well known that the $\NmuD$ spectrum of \eqref{700} is compact and has the form $\Sigma_{\NmuD}(A)=\bigcup_{i=1}^{n}[a_i,b_i]$, for $n\in\mathbb{N}$. From this spectral decomposition we obtain $n+1$ spectral gaps. Specifically, we denote the spectral gaps by $(b_i,a_{i+1})$, with $i=0,1,\ldots,n$, where $b_0=-\infty$ and $a_{n+1}=+\infty$. It is important to emphasize that just two of these spectral gaps are unbounded: $(b_0,a_1)$ and $(b_n,a_{n+1})$. For every $\gamma$ in $(b_i,a_{i+1})$, with $i=0,\ldots,n$, the $\gamma$-weighted system $\eqref{701}$ has $\NmuD$ (with unique invariant projector $s\mapsto\P_{i}(s)$ on every spectral gap $(b_i,a_{i+1})$, $i=0,1,\ldots,n$), and from lemma~2.21, the system has ($\Nmu$,$\epsilon$)-growth with constants $\widehat{K}$ and $a+|\gamma|$. In the particular case of the unbounded spectral gaps, it is known that the projector associated with the interval $(b_0, a_1)$ is $\P = 0$, while for the interval $(b_n, a_{n+1})$ it is $\P=\Id$. Therefore, it has sense to define for every $\gamma\in(b_i,a_{i+1})$, with $i=0,\ldots,n$, the regions of stable and unstable constants for the $\gamma$-weighted system. Let us denote these regions by $\St_{\P_i}^{\gamma}$ and $\Un_{\P_i}^{\gamma}$, respectively. Note that  $\St_{\P_i}^{\gamma}$ and $\Un_{\P_i}^{\gamma}$ are bounded sets: 
\begin{itemize}
    \item For $\gamma\in(b_i,a_{i+1})$, with $i=0,\ldots,n-1$, we get 
    \[
    \Un_{\P_i}^{\gamma}\subset\{(\beta,\nu)\in(0,+\infty)\times[0,+\infty):0\leq\nu<\beta\leq a+|\gamma|+\epsilon\}.
    \]
    \item For $\gamma\in(b_i,a_{i+1})$, with $i=1,\ldots,n$, we get 
    \[
    \St_{\P_i}^{\gamma}\subset\{(\alpha,\theta)\in(-\infty,0)\times[0,+\infty):-(a+|\gamma|+\epsilon)\leq\alpha<\theta< a+|\gamma|+\epsilon\}.
    \]
\end{itemize}

\begin{remark}
    {\rm The boundedness of the stable region and unstable region is uniform in every bounded spectral gap. Indeed, we infer
    \[
    \St_{\P_i}^{\gamma}\subset\{(\alpha,\theta)\in(-\infty,0)\times[0,+\infty):-(a+\epsilon+\tau)\leq\alpha<\theta< a+\tau+\epsilon\},
    \]
    for all $\gamma\in(b_i,a_{i+1})$, with $i=1,\ldots,n-1$, where $\tau=\max\{|b_i|,|a_{i+1}|: i=1,\ldots,n-1\}$. Similarly, we obtain
    \[
    \Un_{\P_i}^{\gamma}\subset\{(\beta,\nu)\in(0,+\infty)\times[0,+\infty):0\leq\nu<\beta\leq a+\tau+\epsilon\},
    \]
    for all $\gamma\in(b_i,a_{i+1})$, with $i=1,\ldots,n-1$. 
    }
\end{remark}

\begin{definition}
{\rm \cite[Def.~2.16]{GJ}}
		Assume the system \eqref{700} admits $(\Nmu,\epsilon)$-growth.    Let $(b_i,a_{i+1})$ be a spectral gap. 
		\begin{itemize}
			\item Let $i=1,\dots,n$. The \textbf{function of optimal stable ratio}  $\st_{\P_{i}}\colon(b_i,a_{i+1})\to \mathbb{R}$ is defined as 
			\[
			\gamma\mapsto \st_{\P_{i}}^\gamma:=\inf\{\alpha+\theta:(\alpha,\theta)\in \St_{\P_{i}}^{\gamma}\}.
			\]		
			\item Let $i=0,\dots,n-1$. The \textbf{function of optimal unstable ratio}  $\un_{\P_{i}}\colon(b_i,a_{i+1})\to \mathbb{R}$ is defined as 
			\[
			\gamma\mapsto \un_{\P_{i}}^\gamma:=\sup\{\beta-\nu:(\beta,\nu)\in \Un_{\P_{i}}^{\gamma}\}.
			\]
		\end{itemize}
	\end{definition}

In the next, we will develop some interesting properties of the optimal --stable and unstable-- ratio maps. In the following statements we are assuming that system \eqref{700} has $(\Nmu,\epsilon)$-growth; see also \cite[Lem.~2.17 \& Prop.~2.18]{GJ}.

\begin{proposition}\label{743}
    On every spectral gap, the map $\un_{\P_{i}}$ --similarly, $\st_{\P_{i}}$-- is decreasing and continuous. 
\end{proposition}

\begin{proof}
    At first, we prove that $\un_{\P_{i}}$ is decreasing. Consider $\zeta,\gamma\in(b_i,a_{i+1})$ with $\zeta>\gamma$. Then we can see $\Un_{\P_{i}}^{\zeta}\subset\Un_{\P_{i}}^{\gamma}$, from which we will immediately deduce $\un_{\P_{i}}^{\zeta}\leq\un_{\P_{i}}^{\gamma}$. Indeed, for $(\beta,\nu)\in\Un_{\P_{i}}^{\zeta}$, there exists $K\geq1$ such that 
    \[
    \|\Phi_{\zeta}(k,n)\{\Id-\P_{i}(n)\}\|\leq K\left(\dfrac{\mu(k)}{\mu(n)}\right)^{\beta}\mu(n)^{\sgn(n)\nu}, \quad \text{ for all $n\geq k$}.
    \]
    On the other hand, for every $k,n\in\mathbb{Z}$ we get 
    \[
\Phi_{\gamma}(k,n)=\Phi_{\zeta}(k,n)\left(\dfrac{\mu(k)}{\mu(n)}\right)^{\zeta-\gamma}.
    \]
Moreover, since $\zeta>\gamma$, for $n\geq k$ we obtain $\left(\dfrac{\mu(k)}{\mu(n)}\right)^{\zeta-\gamma}\leq1$. Therefore, we infer
\begin{equation*}
    \begin{split}
       \|\Phi_{\gamma}(k,n)\{\Id-\P_{i}(s)\}\|&=\left(\dfrac{\mu(k)}{\mu(n)}\right)^{\zeta-\gamma}\|\Phi_{\zeta}(k,n)\{\Id-\P_{i}(n)\}\|\\
       &\leq K\left(\dfrac{\mu(k)}{\mu(n)}\right)^{\beta}\mu(n)^{\sgn(n)\nu}, \quad \text{ for all $n\geq k$},
    \end{split}
\end{equation*}
which in turn implies that $(\beta,\nu)\in\Un_{\P_{i}}^{\gamma}$.

Now we prove $\un_{\P_i}$ is continuous. Since $\un_{\P_{i}}$ is a decreasing function, we know that its lateral limits exist at every $\gamma\in (b_i,a_{i+1})$.

\noindent{\bf Claim 1:} The function $\un_{\P_{i}}$ is right-continuous at $\gamma\in(b_i,a_{i+1})$. By contradiction, suppose that is not the case. Then there exists $\epsilon_1>0$ such that $\un_{\P_{i}}^{\gamma}-\epsilon_1>\un_{\P_{i}}^{\zeta}$, for all $\zeta>\gamma$. Consider $0<\epsilon_2<\epsilon_1/3$, and choose $(\beta^{\gamma},\nu^{\gamma})\in\Un_{\P_{i}}^{\gamma}$ such that 
\begin{equation}\label{724}
    \un_{\P_{i}}^{\gamma}-\epsilon_2<\beta^{\gamma}-\nu^{\gamma}<\un_{\P_{i}}^{\gamma}.
\end{equation}
There exists $K\geq1$ such that 
\begin{equation*}
		\|\Phi_\gamma(k,n)\{\Id-\P_{i}(n)\}\|\leq K\left(\frac{\mu(k)}{\mu(n)}\right)^{\beta^\gamma}\mu(n)^{\sgn(n)\nu^\gamma},\quad \text{ for all $n\geq k$}.
		\end{equation*}
Choose now $\epsilon_3>0$ such that $\epsilon_3<\min\{\beta^\gamma-\nu^\gamma,\epsilon_1/3\}$. For $\gamma<\zeta<\gamma+\epsilon_3$ we have
		\begin{align*}
		\|\Phi_\zeta(k,n)\{\Id-\P_{i}(n)\}\|&\leq K\left(\frac{\mu(k)}{\mu(n)}\right)^{\beta^\gamma+\gamma-\zeta}\mu(n)^{\sgn(n)\nu^\gamma}\\
		&\leq K\left(\frac{\mu(k)}{\mu(n)}\right)^{\beta^\gamma-\epsilon_3}\mu(n)^{\sgn(n)\nu^\gamma} ,\quad \text{ for all $n\geq k$}.
		\end{align*}
Now, since $\beta^\gamma-\epsilon_3-\nu^\gamma>0$, we infer that  $(\beta^\gamma-\epsilon_3,\nu^\gamma)\in \Un_\P^\zeta$. Hence, from our first supposition and definition of $\un_{\P_{i}}^{\zeta}$, we deduce
\[
\un_{\P_{i}}^{\gamma}-\epsilon_1>\un_{\P_{i}}^{\zeta}\geq \beta^\gamma-\epsilon_3-\nu^\gamma.
\]
Moreover, from \eqref{724} we get
\[
\beta^\gamma-\epsilon_3-\nu^\gamma>\un_{\P_{i}}^{\gamma}-\epsilon_2-\epsilon_3>\un_{\P_{i}}^{\gamma}-\epsilon_1, 
\]
from which clearly we obtain a contradiction.

\noindent{\bf Claim 2:} The function $\un_{\P_{i}}$ is left-continuous at $\gamma\in(b_i,a_{i+1})$. By contradiction, suppose that is not the case. Then there exists $\epsilon_1>0$ such that $\un_{\P_{i}}^{\zeta}-\epsilon_1>\un_{\P_{i}}^{\gamma}$, for all $\zeta<\gamma$. Consider $0<\epsilon_2<\epsilon_1/3$ and $\zeta\in(\gamma-\epsilon_2,\gamma)$. Choose $\epsilon_3<\epsilon_1/3$ and a pair $(\beta^{\zeta},\nu^{\zeta})\in\Un_{\P_{i}}^{\zeta}$ such that 
\begin{equation}\label{725}
    \un_{\P_{i}}^{\zeta}-\epsilon_3<\beta^{\zeta}-\nu^{\zeta}<\un_{\P_{i}}^{\zeta}
\end{equation}
Then there exists $K\geq1$ such that .
		\begin{equation*}
		\|\Phi_\zeta(k,n)\{\Id-\P_{i}(n)\}\|\leq K\left(\frac{\mu(k)}{\mu(n)}\right)^{\beta^\zeta}\mu(n)^{\sgn(n)\nu^\zeta},\quad \text{ for all $n\geq k$}.
		\end{equation*}
From which we deduce
		\begin{align*}
		\|\Phi_\gamma(k,n)\{\Id-\P_{i}(n)\}\|&\leq K\left(\frac{\mu(k)}{\mu(n)}\right)^{\beta^\zeta+\zeta-\gamma}\mu(n)^{\sgn(n)\nu^\zeta}\nonumber\\
		&\leq K\left(\frac{\mu(k)}{\mu(n)}\right)^{\beta^\zeta-\epsilon_2}\mu(n)^{\sgn(n)\nu^\zeta} ,\quad \text{ for all $n\geq k$}.
		\end{align*}
Now, by using \eqref{725} and the first supposition, we obtain 
\[
0\leq\un_{\P_{i}}^{\gamma}<\un_{\P_{i}}^{\zeta}-\epsilon_1<\un_{\P_{i}}^{\zeta}-\epsilon_2-\epsilon_3<\beta^\zeta-\epsilon_2-\nu^{\zeta}.
\]
Therefore, the pair $(\beta^\zeta-\epsilon_2,\nu^{\zeta})$ belongs to $\Un_{\P_{i}}^{\gamma}$. However, we have $\beta^\zeta-\epsilon_2-\nu^{\zeta}>\un_{\P_{i}}^{\gamma}$, which contradicts the maximality of  $\un_{\P_{i}}^{\gamma}$. 
\end{proof}

The following result highlights a key property of the optimal ratio maps $\un_{\P_{i}}$ and $\st_{\P_{i}}$. For a complete proof, see \cite[Thm.~2.18 \& Lem.~2.22]{GJ}.

\begin{theorem}\label{limits}
    For every bounded spectral gap $(b_i,a_{i+1})$, i.e. for $i=1,\ldots,n-1$, we have 
    \[
    \lim_{\gamma\to b_i^{+}}\st_{\P_{i}}^{\gamma}=\lim_{\gamma\to a_{i+1}^{-}}\un_{\P_{i}}^{\gamma}=0.
    \]
    Moreover, the following limits hold 
    \begin{itemize}
        \item[(i)] $\displaystyle\lim_{\gamma\to b_{n}^{+}}\st_{\Id}^{\gamma}=\lim_{\gamma\to a_{1}^{-}}\un_{0}^{\gamma}=0$
        \item[(ii)] $\displaystyle\lim_{\gamma\to +\infty}-\st_{\Id}^{\gamma}=\lim_{\gamma\to -\infty}\un_{0}^{\gamma}=+\infty$.
    \end{itemize}
\end{theorem}

\begin{proof} Let us prove $\lim_{\gamma\to -\infty}\un_{0}^{\gamma}=+\infty$. Since $\un_{0}$ is decreasing and continuous, it suffices to show that it cannot be bounded on $(-\infty,a_1)$. By contradiction, assume that $\un_0$ is bounded, {\it i.e.} there exists $M>0$ such that $0<\un_0^{\gamma}<M$, for all $\gamma<a_1$. 

Let $n\geq k$. From the ($\Nmu$,$\epsilon$)-growth, we get 
\[
\|\Phi(k,n)\|\leq\widehat{K}\left(\dfrac{\mu(k)}{\mu(n)}\right)^{-a}\mu(n)^{\sgn(n)\epsilon}.
\]
Now, consider $\gamma<-(M+a+\epsilon)$. From above, we deduce
\[
\|\Phi_{\gamma}(k,n)\|\leq\widehat{K}\left(\dfrac{\mu(k)}{\mu(n)}\right)^{-a-\gamma}\mu(n)^{\sgn(n)\epsilon},\quad \text{ for all $n\geq k$}.
\]
The previous estimation defines a dichotomy with parameters $(0;*,-\gamma-a,*,\epsilon)$, or in other words, the pair $(-\gamma-a,\epsilon)$ belongs to $\Un_0^{\gamma}$. Hence, we infer $-\gamma-a-\epsilon\leq\un_0^{\gamma}$. However, this leads to a contradiction as we have
\[
M<-\gamma-a-\epsilon\leq\un_{0}^{\gamma}<M.
\]
\end{proof}

Figure 1 illustrates an example of spectral decomposition in two compact intervals, {\it i.e.} $\Sigma_{\NmuD}=[a_1, b_1] \cup [a_2, b_2]$, together with the graph of the functions $\un_{\P}$ and $\st_{\P}$ defined on the spectral gaps. 

\begin{figure}[h]
		\centering
		\includegraphics[width=0.70\textwidth]{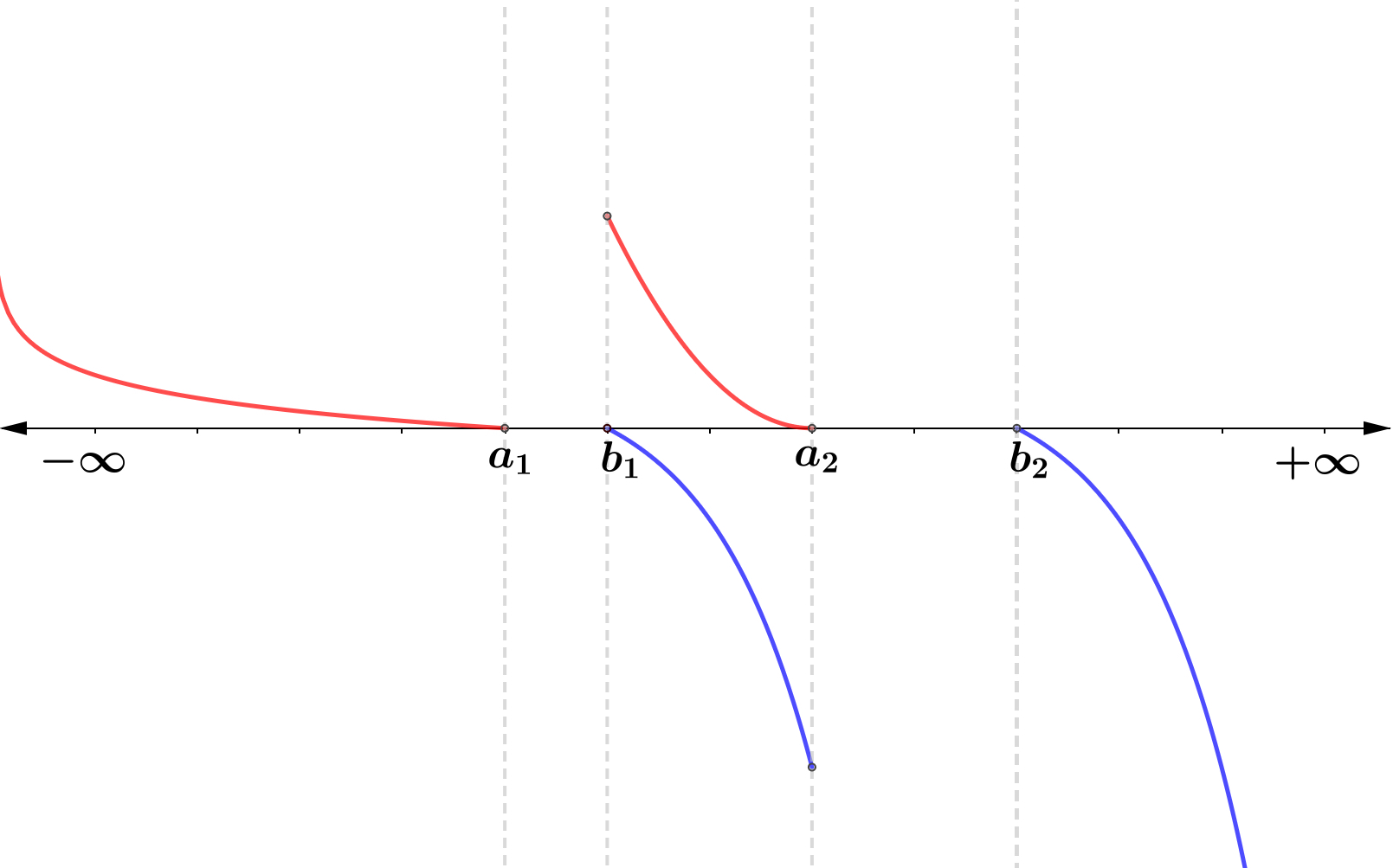}
		\caption{Sketch of the graph of the functions $\un_{\P}$ and $\st_{\P}$ defined on the spectral gaps $(-\infty,a_1)$, $(b_1,a_2)$, and $(b_2,+\infty)$.}
\end{figure}

\begin{remark}\label{boundarypoints}
    {\rm When the system \eqref{700} admits $(\Nmu,\epsilon)$-growth, we know that its nonuniform $\mu$-dichotomy spectrum as the form $\Sigma_{\NmuD}(A)=\cup_{i=1}^{m}[a_i,b_i]$, with $m\in\{1,\ldots,d\}$. Moreover, we can define the optimal ratio maps on every spectral gap, which are decreasing and continuous functions satisfying Thm.~\ref{limits}. Therefore, note that if there are some numbers $a,b\in\mathbb{R}$ such that $\lim_{\gamma\to a^{-}}\un_{\P}^{\gamma}=0$ and $\lim_{\gamma\to b^{+}}\st_{\P}^{\gamma}=0$, then those numbers must be boundary points of $\Sigma_{\NmuD}(A)$, that means, $a$ is exactly one of the $a_i$ and $b$ is exactly one of the $b_i$ from the spectral decomposition, with $i\in\{1,\ldots,m\}$.}
\end{remark}

\section{Kinematic Similarity and Spectra Non-invariance}

In this section, contrary to the uniform case, we will observe that discrete nonautonomous linear systems which exhibit nonuniform kinematic similarity do not preserve the nonuniform dichotomy spectrum.

\subsection{Continuous Nonuniform Kinematic Similarity}
We commence by recalling the concept of $(\mu,\varepsilon)$-kinematic similarity introduced in \cite{Silva}, which is given in the continuous case, {\it i.e.} the map $t \mapsto A(t) \in \mathbb{R}^{d\times d}$ is locally integrable and $t\in\mathbb{R}$.

Let us consider the nonautonomous linear differential equation
\begin{equation}\label{800}
    \dot{x}=A(t)x, \qquad t\in\mathbb{R}, 
\end{equation}
where $t \mapsto A(t) \in \mathbb{R}^{d\times d}$ is locally integrable. 
\begin{definition}\label{KS}
		Let $\mu\colon\mathbb{R}\to(0,+\infty)$ be a growth rate and $t \mapsto B(t)\in\mathbb{R}^{d\times d}$ be a locally integrable map. Given $\varepsilon\geq0$, the system \eqref{800} is \textbf{nonuniformly $(\mu,\varepsilon)$-kinematically similar}, to a system 
  \begin{equation}\label{727}
      \dot{y}=B(t)y, \qquad t\in\mathbb{R},
  \end{equation}
  if there is a constant $M_\varepsilon>0$, and a differentiable matrix function $S\colon\mathbb{R}\to GL_d(\mathbb{R})$ satisfying the following properties:
		\begin{itemize}
			\item[(i)] $\|S(t)\|\leq M_\varepsilon\mu(t)^{\sgn(t)\varepsilon}$, for all $t\in\mathbb{R}$.
			\item[(ii)]$\|S(t)^{-1}\|\leq M_\varepsilon\mu(t)^{\sgn(t)\varepsilon}$, for all $t\in\mathbb{R}$.
			\item[(iii)]If $t\mapsto y(t)$ is a solution for \eqref{727}, then $t\mapsto x(t):=S(t)y(t)$ is a solution for \eqref{800}.
			\item[(iv)] If $t\mapsto x(t)$ is a solution of \eqref{800}, then $t\mapsto y(t):=S(t)^{-1}x(t)$ is a solution of \eqref{727}.
			\end{itemize}
In the particular case that $\varepsilon=0$, it is said that systems \eqref{800} and \eqref{727}  are \textbf{uniformly kinematically similar} or simply \textbf{kinematically similar}.
\end{definition}

Every differentiable matrix function $S\colon\mathbb{R}\to GL_d(\mathbb{R})$ satisfying (i) and (ii) for some $\varepsilon\geq0$ is called a nonuniform Lyapunov matrix function with respect to $\mu$ and the transformation of coordinates $y(t)=S(t)^{-1}x(t)$ is said a nonuniform Lyapunov transformation with respect to $\mu$.

\begin{lemma}
		{\rm \cite[Lemma 13]{Silva}} Let $S\colon\mathbb{R}\to GL_d(\mathbb{R})$ be a nonuniform Lyapunov matrix function with respect to $\mu$ for some $\varepsilon\geq0$. Then, the following statements are equivalent:
		\begin{itemize}
			\item [(a)] $S(t)$ verifies {\rm(iii)} and {\rm (iv)} from Definition \ref{KS}.
			\item [(b)] The identity $\Phi(t,s)S(s)=S(t)\Psi(t,s)$ holds, for all $t,s\in \mathbb{R}$, where $\Phi$ and $\Psi$ are the evolution operators of \eqref{800} and \eqref{727}, respectively.
			\item [(c)] $S(t)$ is a solution of $\dot{S}=A(t)S-SB(t)$.
		\end{itemize}
	\end{lemma}

\begin{remark}
{\rm Considering that $\mu(t)=e^{t}$, for all $t\in\mathbb{R}$, the above Def.~\ref{KS} coincide with the notion of nonuniform kinematic similarity defined and used in literature so far, see e.g. \cite{Chu,Xiang}. In addition, in the uniform case $\varepsilon=0$, coincide with the classical notion of kinematic similarity, see for instance \cite{Siegmund2}.}
\end{remark}

In the autonomous case, we can think of kinematic similarity as a property that in some cases simplifies a linear system $\dot{x} = Ax$ into a more convenient form, such as an uncoupled or a block diagonalized system, see \cite[Subsec.~1.2]{Perko}. On the other hand, in the nonautonomous framework, a classical example of kinematic similarity is provided by Floquet's theorem \cite{Floquet}. Indeed, consider the system \eqref{700} with $t \mapsto A(t)\in\mathbb{R}^{d\times d}$ continuous and $\omega$-periodic. The map $S:\mathbb{R} \to GL_d(\mathbb{R})$ defined by $S(t) = X(t)e^{-Qt}$, for all $t \in \mathbb{R}$, is $\omega$-periodic and a Lyapunov transformation, where $X(t)$ is a fundamental matrix associated to \eqref{800} with $X(0) = \Id$, and $Q = \frac{1}{\omega} \ln X(\omega)$, known in the literature as the monodromy matrix. In this case, the nonautonomous system \eqref{800} is kinematically similar to the autonomous linear system $\dot{y} = Qy$.

\subsection{Discrete Nonuniform Kinematic Similarity} In contrast to the continuous case, as far as we know, the only notion of discrete nonuniform kinematic similarity available in literature is stated in \cite{Chu2}. This definition differs from the continuous case due to a particular condition over the nonuniformity. For this reason, the authors called it {\it weak kinematic similarity}. 

To introduce the concept of weak kinematic similarity from \cite{Chu2}, we recall that the notion of nonuniform exponential dichotomy discussed in that article is expressed as in equation \eqref{728}, which is a particular case of Def.~\ref{DefNmuD} with $\mu(n)=e^{n}$, $a=e^{\alpha}$, $\beta=-\alpha$, $\theta=\ln(\varepsilon)$, $\nu=\theta$ and $\alpha+2\theta<0$, see Rem.~\ref{729} and \ref{730}.

Let us consider the nonautonomous linear difference equation
\begin{equation}\label{801}
    y(k+1)=B(k)y(k), \qquad k\in\mathbb{Z}. 
\end{equation}
 
\begin{definition}\label{weaknondegenerate}{\rm (\cite[Def.~3.3]{Chu2})}
    The map $S:\mathbb{Z}\to GL_{d}(\mathbb{R})$ is called \textbf{weakly nondegenerate} if there exists a constant $M=M_\varepsilon>0$ such that
    \[
    \|S(k)\|\leq M\varepsilon^{|k|} \quad \text{ and } \quad \|S(k)^{-1}\|\leq M\varepsilon^{|k|}, \text{ for all $k\in\mathbb{Z}$}, 
    \]
    where $\varepsilon$ is the same constant that in \eqref{728}.
\end{definition}

\begin{remark}\label{748}
    {\rm Note that $M\varepsilon^{|k|}$ can be simply seen as $M\mu(k)^{\sgn(k)\theta}$, where $\mu$ is the exponential map. The consideration of the parameter $\theta$ being the same as the error in the nonuniform $\mu$-dichotomy significantly differs from the continuous case of a nonuniform Lyapunov transformation, where the parameter is not necessarily the same as the error in the nonuniform dichotomy.}  
\end{remark}

\begin{definition}
The system \eqref{700} is \textbf{weakly kinematically similar} to \eqref{801} if there exists a weakly nondegenerate matrix function $S$ such that 
\[
S(k+1)B(k)=A(k)S(k), \qquad \text{for all $k\in\mathbb{Z}$.}
\]
\end{definition}

In the next statement, it is considered the growth rate $\mu(n)=e^{n}$, for all $n\in\mathbb{Z}$. 

\begin{corollary}{\rm (\cite[Cor.~3.9]{Chu2})}\label{749}
    Assume that system \eqref{700} is weakly kinematically similar to \eqref{801}. Then $\Si_{\NmuD}(A)=\Si_{\NmuD}(B)$. 
\end{corollary}

In what follows, we present an example that contradicts the above result. Elucidating that the nonuniform exponential dichotomy spectrum is not invariant under weak kinematic similarity as asserted in \cite{Chu2}. We emphasize that in the next example, we are considering the conditions $\alpha+\theta<0$ and $\beta-\nu>0$ from Def.~\ref{DefNmuD}. However, even considering $\alpha+2\theta<0$ as assumed in \cite{Chu2}, we will see that the result is the same: the spectrum of a system is not preserved under weak kinematic similarity. 

\begin{example}\label{773}
    {\rm Let $3a>\omega>a$ and consider the difference equation
    \begin{equation}\label{731}
        x(n+1)=A(n)x(n), \qquad n\in\mathbb{Z},
    \end{equation}
 where $A(n)=e^{-\omega+a(n+1)\cos(n+1)-an\cos(n)-a\sin(n+1)+a\sin(n)}$. 
 
 The transition matrix of \eqref{731} is given by 
 \[
 \Phi(k,n)=e^{-\omega(k-n)+ak\cos(k)-an\cos(n)-a\sin(k)+a\sin(n)}, \qquad k,n \in\mathbb{Z}.
 \]

\noindent{\bf Claim 1:} The system \eqref{731} admits nonuniform exponential bounded growth with the estimation
 \begin{equation}\label{734}
     |\Phi(k,n)|\leq e^{2a}e^{(\omega+a)|k-n|}e^{2a|n|}, \quad \text{for all $k,n\in\mathbb{Z}$}.
 \end{equation}

Let us commence by noting that we can write the transition matrix in the following form
 \[
 \Phi(k,n)=e^{(-\omega+a)(k-n)+ak(\cos(k)-1)-an(\cos(n)-1)}e^{-a\sin(k)+a\sin(n)}, \quad k,n\in\mathbb{Z}.
 \]
In order to deduce \eqref{734}, we will prove two intermediate estimations for $|\Phi(k,n)|$. First, we analyse the cases for $k\geq n$. 
\begin{itemize}
    \item [(i)] Case $k\geq n\geq0:$ since $ak(\cos(k)-1)\leq0$, we get 
\[
|\Phi(k,n)|\leq e^{2a}e^{(-\omega+a)(k-n)}e^{-an(\cos(n)-1)}.
\]
Moreover, since $-2\leq\cos(n)-1$, we obtain that $2an\geq -an(\cos(n)-1)$, which in turn implies 
\[
|\Phi(k,n)|\leq e^{2a}e^{(-\omega+a)(k-n)}e^{2an}=e^{2a}e^{(-\omega+a)(k-n)}e^{2a|n|}.
\]

\item [(ii)] Case $k\geq 0\geq n:$ since $ak(\cos(k)-1)\leq0$ and $-an(\cos(n)-1)\leq0$, we obtain 
\begin{equation*}
    \begin{split}
        |\Phi(k,n)|&\leq e^{2a}e^{(-\omega+a)(k-n)}\\
        &\leq e^{2a}e^{(-\omega+a)(k-n)}e^{2a|n|}.
    \end{split}
\end{equation*}

\item [(iii)] Case $0\geq k\geq n:$ using a similar argument as in the previous cases, we infer 
\[
|\Phi(k,n)|\leq e^{2a}e^{(-\omega+a)(k-n)}e^{ak(\cos(k)-1)}.
\]
Moreover, since $-2\leq(\cos(k)-1)$, we obtain that $-2ak\geq ak(\cos(k)-1)$, which implies that
\begin{equation*}
    \begin{split}
        |\Phi(k,n)|&\leq e^{2a}e^{(-\omega+a)(k-n)}e^{ak(\cos(k)-1)}\\
        &\leq e^{2a}e^{(-\omega+a)(k-n)}e^{-2ak}\\
        & = e^{2a}e^{(-\omega+a)(k-n)}e^{2a|k|}\\
        & \leq e^{2a}e^{(-\omega+a)(k-n)}e^{2a|n|}.
    \end{split}
\end{equation*}
\end{itemize}
From (i), (ii) and (iii), we conclude 
\begin{equation}\label{732}
|\Phi(k,n)|\leq e^{2a}e^{(-\omega+a)(k-n)}e^{2a|n|}, \qquad k\geq n.
\end{equation}

On the other hand, to analyze the cases for $n\geq k$, we use the fact that the transition matrix can be written in the form 
\[
 \Phi(k,n)=e^{(-\omega-a)(k-n)+ak(\cos(k)+1)-an(\cos(n)+1)}e^{-a\sin(k)+a\sin(n)}, \quad k,n\in\mathbb{Z}.
 \]
 
Considering the above decomposition, together with  the fact that $2\geq(\cos(n)+1)\geq 0$, for all $n\in\mathbb{Z}$, we infer
\begin{equation}\label{733}
    |\Phi(k,n)|\leq e^{2a}e^{(-\omega-a)(k-n)}e^{2a|n|}, \qquad n\geq k.
\end{equation}

Therefore, gathering the estimations \eqref{732} and \eqref{733}, it is straightforward to deduce the inequality \eqref{734}. 

\vspace{0.1cm}

\noindent{\bf Claim 2:} The nonuniform exponential dichotomy spectrum of system \eqref{731} is given by
\[
\Sigma_{\NmuD}(A)=[-\omega-3a,-\omega+3a].
\] 
In fact, since the one-dimensional system \eqref{731} has $(\Nmu,\epsilon)$-growth, its nonuniform dichotomy spectrum is a nonempty compact interval of the form $\Sigma_{\NmuD}(A)=[c,d]$. In what follows, we will find exactly the values $c$ and $d$.  

From \eqref{732} we can estimate the $\gamma$-weighted operator as
\[
|\Phi_{\gamma}(k,n)|\leq e^{2a}e^{(-\omega-\gamma+a)(k-n)}e^{2a|n|}, \qquad k\geq n.
\]
Then, for every $\gamma$ such that $\gamma>-\omega+3a$, we deduce that $\gamma$ belongs to the nonuniform resolvent of \eqref{731} and $\st_{\P}^{\gamma}=-\omega-\gamma+3a$. Additionally, from Rem.~\ref{boundarypoints}, we infer that $d=-\omega+3a$.

Similarly as above, from \eqref{733} we derive the inequality
\[
 |\Phi_{\gamma}(k,n)|\leq e^{2a}e^{(-\omega-\gamma-a)(k-n)}e^{2a|n|}, \qquad n\geq k.
\]
Hence, for every $\gamma$ such that $\gamma<-\omega-3a$, we deduce that $\gamma$ belongs to the nonuniform resolvent of \eqref{731} and $\un_{\P}^{\gamma}=-\omega-\gamma-3a$. Additionally, from Rem.~\ref{boundarypoints}, we infer that $c=-\omega-3a$.

\noindent{\bf Claim 3:} The system \eqref{731} is weakly kinematically similar to 
\begin{equation}\label{735}
        x(n+1)=B(n)x(n), \qquad n\in\mathbb{Z},
\end{equation}
where $B(n)=e^{-(\omega-a)+a(n+1)\cos(n+1)-an\cos(n)-a\sin(n+1)+a\sin(n)}$.  

Indeed, the map $S\colon\mathbb{Z}\to\mathbb{R}$ defined by $S(n)=e^{-an}$ is a weakly nondegenerate matrix function. Note that the estimations $|S(n)|\leq e^{2a|n|}$ and $|S(n)^{-1}|\leq e^{2a|n|}$ are satisfied, for all $n\in\mathbb{Z}$. In addition, the term $e^{2a}$ coincide exactly with the nonuniform error from \eqref{732} --as required in Def.~\ref{weaknondegenerate}-- and  for every $n\in\mathbb{Z}$ we have 
\begin{equation*}
\begin{split}
 S(n+1)B(n)&=e^{-a(n+1)}e^{-(\omega-a)+a(n+1)\cos(n+1)-an\cos(n)-a\sin(n+1)+a\sin(n)}\\
    &=e^{-\omega+a(n+1)\cos(n+1)-an\cos(n)-a\sin(n+1)+a\sin(n)}e^{-an}\\
    &=A(n)S(n).   
\end{split}
\end{equation*}

\noindent{\bf Claim 4:} The nonuniform exponential dichotomy spectrum of system \eqref{735} is given by 
\[
\Sigma_{\NmuD}(B)=[-\omega-2a,-\omega+4a].
\]

This assertion follows directly from the development established in Claim 2. Note that $\omega$ is just replaced by $(\omega-a)$. 

\vspace{0.2cm}

\noindent{\bf Conclusion:} Summarising what was done in the previous claims, the system \eqref{731} is weakly kinematically similar to system \eqref{735}, but 
\[
\Sigma_{\NmuD}(A)=[-\omega-3a,-\omega+3a]\neq[-\omega-2a,-\omega+4a]=\Sigma_{\NmuD}(B).
\]
}

\end{example}

\bigskip

\subsection{The Non-invariance Problem}\label{774}

In the discrete context, the additional hypothesis in the definition of weak kinematic similarity —regarding the error parameter being the same as the dichotomy error— does not solve the problem of spectral non-invariance detected in \cite{GJ}. 

For simplicity, in the remainder of this section, when we write that a system admits $\NmuD$, we are assuming that $\theta = \nu$ in Def.~\ref{DefNmuD}. Moreover, we will also assume that the bound in Def.~\ref{weaknondegenerate} is given by $M\mu(k)^{\sgn(k)\theta}$, where $\mu\colon\mathbb{Z}\to\mathbb{R}^{+}$ is a discrete growth rate (see Rem.~\ref{748}). Let us, in this case, refer to it as $\mu$-weak kinematic similarity to emphasize the growth rate $\mu$. 

Similarly to \cite[Lem.~3.5]{GJ}, we can establish the following result:

\begin{lemma}\label{751}
    Assume that system \eqref{700} admits $\NmuD$ with parameters $(\P;\alpha,\beta,\theta,\theta)$ and is $\mu$-weakly kinematically similar to \eqref{801} and 
    \begin{equation}\label{747}
        \min\{-\alpha,\beta\}>4\theta.
    \end{equation}
    Then system \eqref{801} admits $\NmuD$ with parameters $(\Q; \alpha+\theta,\beta-\theta,3\theta,3\theta)$, where $\Q(n)=S^{-1}(n)\P(n)S(n)$, for all $n\in\mathbb{Z}$. 
\end{lemma}

On the other hand, we can establish the next result

\begin{lemma}\label{752}
    Let $\mu\colon\mathbb{Z}\to\mathbb{R}^{+}$ be a discrete growth rate and let $\gamma\in\mathbb{R}$. Assume that system \eqref{700} is $\mu$-weakly kinematically similar to the system \eqref{801}. Then the $\gamma$-weighted systems \eqref{700} and \eqref{801} are  $\mu$-weakly kinematically similar.
\end{lemma}
\begin{proof}
    This follows directly from definition of weak kinematic similarity. We emphasize that the $\gamma$-weighted systems \eqref{700} and \eqref{801} are $\mu$-weakly kinematically similar with the same weakly nondegenerate map $S\colon\mathbb{Z}\to GL_{d}(\mathbb{R})$ for which systems \eqref{700} and \eqref{801} are $\mu$-weakly kinematically similar.
\end{proof}

Let us return to the Cor.~\ref{749} in which the invariance of the $\NmuD$-spectrum through $\mu$-weak kinematic similarity is asserted. Note that this result is equivalent to say that $\NmuD$-resolvent of the systems are the same through $\mu$-weak kinematic similarity, {\it i.e.} $\rho_{\NmuD}(A)=\rho_{\NmuD}(B)$. The main issue with this assertion is revealed in Lem.~\ref{751}, in which condition \eqref{747} is neccessary in order to obtain an inclusion of the resolvent sets, and otherwise, condition \eqref{747} will strictly depend on the choice of $\gamma$ in $\rho_{\NmuD}(A)$. We will delve into this idea to clarify it in detail: assume that system \eqref{700} is $\mu$-weakly kinematically similar to the system \eqref{801} and let $\gamma\in\rho_{\NmuD}(A)$. Then $\gamma$-weighted system \eqref{700} admits $\NmuD$ with parameters $(\P_{i};\alpha,\beta,\theta,\theta)$, for some invariant projector $n\mapsto\P_{i}(n)$. Additionally, from Lem.~\ref{752}, we know that the $\gamma$-weighted systems \eqref{700} and \eqref{801} are $\mu$-weakly kinematically similar. Now, in order to ensure that $\gamma$ belongs to $\rho_{\NmuD}(B)$ we need that condition \eqref{747} holds (the dependence of this condition on $\gamma$ is evident because $\alpha$, $\beta$ and $\theta$ depend of it), in this case, the $\gamma$-weighted system \eqref{801} admits $\NmuD$ with parameters $(\Q _{i};\alpha+\theta,\beta-\theta,3\theta,3\theta)$, where $\Q _{i}=S^{-1}\P _{i} S$.

In the context of the above situation, the optimal ratio maps defined in Sect. \ref{optdichconst} provide a partial answer to the non-invariance problem. They also ensure that non-invariance is always detected in the specific case where system \eqref{700} exhibits $(\Nmu,\epsilon)$-growth. Indeed, assume that system \eqref{700} admits $(\Nmu,\epsilon)$-growth and is $\mu$-weakly kinematically similar to \eqref{801}. Then, for every $\theta>0$, since Thm.~\ref{limits}, we can choose $\gamma$ close enough to the $\NmuD$-spectrum of \eqref{700} such that $\st_{\P_i}^{\gamma} + 3\theta>0$ and $\un_{\P_i}^{\gamma} + 3\theta<0$. In consequence, 
\begin{itemize}
    \item For every $(\alpha,\theta)\in \St_{\P_i}^{\gamma}$, we have $\alpha+\theta\geq \st_{\P_i}^{\gamma}$, which in turn implies that $\alpha +4\theta >0$. Hence, the pair $(\alpha+\theta,3\theta)$ does not belong to $\St_{\Q_{i}}^{\gamma}$;

    \smallskip

    \item  For every $(\beta,\nu)\in \Un_{\P_i}^{\gamma}$, we have $\beta-\theta\leq \un_{\P_i}^{\gamma}$, which in turn implies that $\beta-4\theta <0$. Hence, the pair $(\beta-\theta,3\theta)$ does not belong to $\Un_{\Q_{i}}^{\gamma}$.
\end{itemize}

In summary, under the assumption of $(\Nmu,\epsilon)$-growth, and for $\gamma$ sufficiently close to the $\NmuD$ spectrum of \eqref{700}, Lemma~\ref{751} does not permit us to conclude that $\gamma$ belongs to $\rho_{\NmuD}(B)$.

\begin{remark}
    {\rm
An important consequence of the results developed here is that the reducibility result presented in \cite[Theorem 3.11]{Chu2}, which was later used in \cite{Song}, is imprecise. The reason is that, in general, the spectrum is not preserved by a weak kinematic similarity as stated in \cite[Corollary 3.9]{Chu2}. Based on the observations made in this work, we conclude that two systems that are weakly kinematically similar do not necessarily share any of their spectra.

In particular, we cannot generally conclude that the $\NmuD$-spectrum is shared between a system and its block diagonalization. However, we conjecture that there may be some shared characterization between these systems. Specifically, we propose as a final conjecture that the $\Sigma_\sNmuD^\UPP$ spectrum is indeed shared between a system and its block diagonalization, as obtained through the method developed in \cite[Theorem 3.10]{Chu2}. This spectrum, unlike $\Sigma_\NmuD$, does not depend on the explicit relationship between the parameters, but rather only on the presence of the $\UPP$ property.
    }
\end{remark}


\end{document}